\documentclass[a4]{amsart}
\usepackage{amsmath,amssymb,amsthm,bm}
\usepackage[monochrome]{color}
\bmdefine{\ba}{a}
\bmdefine{\be}{e}
\bmdefine{\bg}{g}
\bmdefine{\bk}{k}
\bmdefine{\bm}{m}
\bmdefine{\bp}{p}
\bmdefine{\bs}{s}
\bmdefine{\bt}{t}
\bmdefine{\bv}{v}
\bmdefine{\bw}{w}
\bmdefine{\by}{y}
\bmdefine{\bz}{z}
\newcommand{\R}{{\mathbb R}}
\newcommand{\Q}{{\mathbb Q}}
\newcommand{\Z}{{\mathbb Z}}

\newcommand{\D}{{\mathcal D}}
\newcommand{\N}{{\mathbb N}}
\newcommand{\A}{{\mathcal A}}
\newcommand{\B}{{\mathcal B}}
\renewcommand{\L}{{\mathbb L}}
\newcommand{\LL}{{\mathcal L}}
\newcommand{\X}{{\mathcal X}}
\newcommand{\Y}{{\mathcal Y}}
\newcommand{\x}{\mathbf{x}}
\newcommand{\y}{\mathbf{y}}
\newcommand{\lxe}{\le_{\mathrm{lex}}}
\newcommand{\lx}{<_{\mathrm{lex}}}
\newtheorem{theorem}{Theorem}[section]
\newtheorem{rem}{Remark}[section]
\newtheorem{lem}[theorem]{Lemma}
\newtheorem{cor}[theorem]{Corollary}
\newtheorem{prop}[theorem]{Proposition}
\newtheorem{ex}[theorem]{Example}
\makeatletter
  
  \@addtoreset{equation}{section}
\makeatother
\begin{document}
\title{Multiplicative analogue of Markoff-Lagrange spectrum and Pisot numbers}
\author{Shigeki Akiyama and %Teturo Kamae, 
Hajime Kaneko}
\maketitle
\begin{abstract}
Markoff-Lagrange spectrum uncovers exotic topological properties of 
Diophantine approximation.
We investigate asymptotic properties of geometric progressions modulo one and observe 
significantly analogous results on the set

\[
{\mathcal L}(\alpha)=\left\{\left.\limsup_{n\to \infty}\|\xi \alpha^n\|\ \right|\ \xi\in {\mathbb R}\right\},
\]
where $\|x\|$ is the distance from $x$ to the nearest integer. 
First, we show that ${\mathcal L}(\alpha)$ is closed in $[0,1/2]$ for any Pisot number $\alpha$.

Then we consider the case where $\alpha$ is an integer with $\alpha\geq 2$, 
or a quadratic unit with $\alpha\ge 3$. 
We show that ${\mathcal L}(\alpha)$ 
contains a proper interval when $\alpha$ is quadratic but it does not when $\alpha$
is an integer. 
We also determine the minimum limit point and all isolated points beneath this point.
In the course of the proof, we revisit a property studied by Markoff which
characterizes bi-infinite balanced words and sturmian words.
\end{abstract}

\section{Introduction}
\subsection{Background}
For a real number \(x\), we write the distance from \(x\) to the nearest integer by 
\(\|x\|\).  
%Let us review some results related to the set $\LL(\alpha)$.
Denote the integral and fractional 
parts of \(x\) by \(\lfloor x\rfloor\) and \(\{x\}\), respectively. 
Moreover, we write the minimal integer greater than or equal to \(x\) by 
\(\lceil x \rceil\). 
In this paper, \(\N\) denotes the set of positive integers.

Hurwitz 
theorem reads for any irrational $x$, there exist
infinitely many rational numbers $p/q$ with
$$
\left| x-\frac{p}{q} \right| \ < \ \frac 1{\sqrt{5}q^2},
$$
where the constant $\sqrt{5}$ can not be replaced by a greater one if and only if $x$ is the image of M\"obius transformation in $GL(2,\Z)$
of the golden mean. %For each irrational $x$, 
The best possible constant $\mu(x)$ related to the Diophantine approximation above is defined by 
\begin{align*}
\sup\left\{
c>0 \ \left| \ \left| x-\frac{p}{q} \right| < \frac{1}{cq^2}\mbox{ has infinitely many rational solutions $p/q$}
\right.
\right\},
\end{align*}
that is, 
\begin{align*}
\mu(x)=\left(\liminf_{n\to \infty} n \|n x\|\right)^{-1}\in [\sqrt{5},\infty].
\end{align*}
We see $\mu((1+\sqrt{5})/2)=\sqrt{5}$, and $\mu(x)< \infty$ if and only if $x$ is badly approximable, namely, 
$x$ has bounded partial quotients in its regular continued fractions. 
Lagrange spectrum\footnote{Markoff spectrum is an intimately related set, 
but we do not deal with in this article. See \cite{Cusick-Flahive:89}.}, defined by
$$
\L=\left \{\left. \mu(x) \ \right| x\in \R \backslash \Q, \mu(x)<\infty \right\},
$$
attracted a lot of attention through rich connection with many areas of mathematics. 
%Lagrange spectrum gives the rate of approximation of the origin 
%$0$ by the sequence $(n \alpha \pmod{1})_{n=1,2,\dots}$. 
Basic results on $\L$ %Lagrange spectrum
include
\begin{itemize}
\item[(i)] $\L$ is closed in $\R$ (Cusick \cite{Cusick:75}).
\item[(ii)] $\L \hspace{0.4mm} \cap \hspace{0.4mm} [0,3)$ is discrete and $3$ is 
the minimum accumulation point of $\L$. Moreover, 
$\L \hspace{0.4mm} \cap \hspace{0.4mm} [0,3)$ is described in terms of Markoff numbers (Markoff \cite{Markoff:80}).
\item[(iii)] Hall \cite{Hall} showed that $\L \supset [6,\infty)$. Moreover, Freiman \cite{Freiman} determined the 
minimum $c$ with $\L \supset [c,\infty)$, where 
\[
c=\frac{2221564096+283748\sqrt{462}}{491993569}=4.527829566\ldots.
\]
\end{itemize}
Moreover, Moreira \cite{Moreira.18} proved that 
the Hausdorff dimension of $\L \cap [0,t]$ is continuous (but not H\"{o}lder continuous) on $t$.  
See further developments in \cite{Cusick-Flahive:89, Bombieri:07, Matheus}. 

Various generalizations has been studied, we just quote a few of them.
From dynamical point of view, 
%For instance, %$\L$ is related to dynamical systems because (\ref{dynam}) 
%means the reciprocal $(\liminf_{n\to\infty}  n \|R_{\alpha}^{n}(0)\|)^{-1}$ of recurrence rate, 
%where $R_{\alpha}^{n}$ denotes the $n$th iterates of  irrational rotation $R_{\alpha}$ with angle $\alpha$. 
%Lower Boshernitzan-Lagrange spectrum, 
%related to the dynamical systems of interval exchange map, was introduced by Ferenczi \cite{} and 
%further developed by Hubert, Marchese and Ulcigrai \cite{}, and Boshernitzan, Delecroix \cite{}. 
%Lagrange spectrum related to a dynamical systems was also widely studied.  
let $M$ be a surface and $\varphi:M\to M$ be a $C^2$-diffeomorphism with a horseshoe $\Lambda$. 
For a $C^2$ function $f: M\to \R$, let $\L_{\varphi,f}$ be defined by 
$$
\L_{\varphi,f}=\left\{\left.\limsup_{n\to\infty} f(\varphi^n(x)) \ \right| \ x\in \Lambda\right\}.
$$
Cerqueira, Matheus, and Moreira \cite{Cer-Mat-Mor} investigated the continuity of the Hausdorff dimension of 
$\L_{\varphi,f}\cap (-\infty,t)$. $\L_{\varphi,f}$ generalizes $\L$ because $\L$ has a similar formulation (see also Remark \ref{CusickAlt}).
From number theoretical interest, 
quadratic Lagrange spectrum is introduced to study 
approximation of real numbers by certain quadratic irrational numbers 
(see Parkkonen and Paulin \cite{Parkkonen-Paulin.11,Parkkonen-Paulin.14}, Bugeaud \cite{Bugeaud.14}, and Pejkovi\'{c} \cite{Pejkovic.16}).
%Lagrange spectrum related to inhomogeneous approximation $\|n\alpha +\eta\|$ has also been studied (see).  

In this paper, we provide yet another type of Lagrange spectrum arose from uniform distribution 
theory. Classical Lagrange spectrum discusses distribution of arithmetic progressions modulo one. 
We replace it by geometric progressions modulo one and try to observe analogy of Lagrange spectrum for the set 
\begin{align*}
\LL(\alpha)=\left\{\left.\limsup_{n\to \infty}\|\xi \alpha^n\|\ \right|\ \xi\in \R\right\}.
\end{align*}
%While the set $\L$ concerns the rate of approximation of the point $0$ 
%by $(n\alpha)_{n=1,2,\dots}$ up to integer differences, 
%the study of $\LL(\alpha)$ is focused on 
%the rate of approximation of $0$ by the sequence $(\xi \alpha^n)_{n=1,2,\dots}$.
%Two settings are not quite identical but of similar flavor, and
%we observe striking analogy between them when $\alpha$ is an integer or a quadratic unit.
However, treating general $\alpha>1$ seems to be too difficult at this stage and 
we restrict ourselves to a narrow class of algebraic numbers.

\subsection{Lagrange spectrum of geometric progressions}
Koksma's theorem \cite[Theorem 4.3]{Kuipers-Niederreiter:74} 
implies that for  a fixed $\alpha>1$ and almost all $\xi\in \R$ (or for a fixed $\xi\neq 0$ and almost all $\alpha>1$), 
the sequence
$(\xi \alpha^n)_{n=0,1,\dots}$ is uniformly distributed modulo one. 
In particular, $\limsup_{n\to\infty}\|\xi\alpha^n\|=1/2$. 
However, for a concrete pair $(\xi, \alpha)$,
it is usually difficult to determine $\limsup_{n\to\infty}\|\xi\alpha^n\|$. %get a non-trivial 
%result on the distribution of $(\xi \alpha^n)_{n=0,1,\dots}$ modulo one.
%For example, we consider the limit points of the fractional parts. 
%Pisot \cite{Pisot:38}, Vijayaraghavan \cite{Vijayaraghavan:40}, and R\'{e}dei \cite{Redei} independently showed for any coprime integers 
%$a,b$ with $a>b\geq 2$ that the sequence $\{(a/b)^n\}$ ($n=0,1,\ldots$) has infinitely many limit points. 
%However, it is unknown whether $\{(a/b)^n\}$ ($n=0,1,\ldots$) is dense in $[0,1]$. 
In the case when $\alpha$ is transcendental, then it is generally difficult to prove that 
$\limsup_{n\to\infty}\|\xi\alpha^n\|>0$. 
%the sequence $\{\alpha^n\}$ ($n=1,2,\ldots$) has at least two limit points. 
%Note that it is still not proved whether $\{e^n\}$ ($n=0,1,\ldots$) has a positive limit point, that is, 
%$\limsup_{n\to\infty}\{e^n\}>0$. 
%we do not know yet if $\lim_{n\to \infty} \{e^n\}\neq 0$ holds, 
%nor $(\{(3/2)^n\})_{n=1,2,\dots}$
%is dense in $[0,1]$. 

The situation is 
better when $\alpha$ is a Pisot number. 
Recall that a Pisot number is an algebraic integer greater than 1 
such that the absolute values of 
the conjugates (except for itself) are less than 1. 
Note that every rational integer greater than 1 is a Pisot number.  
Let \(\alpha\) be a Pisot number. 
Since the trace of the number \(\alpha^n\) is an integer, we get 
$
\lim_{n\to\infty}\| \alpha^n\|=0.
$
Hardy \cite{Hardy} (c.f. \cite{Pisot:38}) proved for any algebraic number \(\alpha\) greater 
than 1 and nonzero real number \(\xi\) that if 
$
\lim_{n\to\infty} \|\xi\alpha^n\| =0, 
$
then \(\alpha\) is a Pisot number and \(\xi\in\mathbb{Q}(\alpha)\). 
%It seems difficult to study $\|\xi\alpha^n\|$ for a general $\xi\in \R$.
% 
% Following seems to contradict with Dubickas \cite{dub2}
%
%For example, letting $p\ge 2$ be an integer, 
%it is an open question to show that for each irrational number \(\xi\) that 
%\begin{eqnarray*}
%\limsup_{n\to\infty}\|\xi p^n\|=\frac12.
%\end{eqnarray*}
%Now we study the maximal limit point of the sequence \(\|\xi\alpha^n\|\) 
%in the case where \(\alpha\) is a Pisot unit and \(\xi\not \in\mathbb{Q}(\alpha)\). \par

However, little is known on the exact value $\limsup_{n\to\infty}\|\xi\alpha^n\|$ in the case where 
$\limsup_{n\to\infty}\|\xi\alpha^n\|>0$ even when 
$\alpha$ is a Pisot number. 
We introduce a result by Dubickas \cite{Dubickas:06_2} in the case where $\alpha=a\geq 2$ is an integer.
Let 
\begin{align}\label{inf_prod}
{{E}}(X):=\frac{1-(1-X)\prod_{n=0}^{\infty}(1-X^{2^n})}{2X}
\end{align}
and $\eta_0:=a^{-1}E(a^{-1})\not\in \mathbb{Q}$. Then he showed that 
$\limsup_{n\to\infty}\|\xi a^n\|\geq \eta_0$ for any irrational $\xi$, and $\limsup_{n\to\infty} \|\eta_0 a^n\|=\eta_0$. 
In view of \cite{Nishioka}, the values of ${{E}}(X)$ at any algebraic numbers $\alpha$ with $0<|\alpha|<1$ are transcendental, 
which gives nice contrast to our results in \S 4.

In this paper we investigate $\limsup_{n\to\infty}\|\xi\alpha^n\|$ \textcolor{red}{for a Pisot number $\alpha$}
by studying the geometrical property of $\LL(\alpha)$. 
We mainly consider %specify
the two simplest cases: $\alpha$ is an integer $\alpha=a\ge 2$ or a quadratic unit
$$
\alpha=\begin{cases}
(b+\sqrt{b^2-4})/2 & 3\le b \in \Z\\
(b+\sqrt{b^2+4})/2 & 1\le b \in \Z.
\end{cases}
$$

\subsection{Results and organization of the paper} 
In \S \ref{SD} we prepare our basic setup 
and show {\bf an exact analogy of (i)} that $\LL(\alpha)$
is closed in Theorem \ref{Closed} for any Pisot number $\alpha$. 
The idea of the proof is to select a bi-infinite word whose all
central blocks reappear infinitely many times in the limiting process, and 
glue the blocks together to make another infinite word without changing the limsup value.
The method also gives a short alternative proof of 
the result by Cusick \cite{Cusick:75} that $\L$ is closed (see Remark \ref{CusickAlt}). 
In \S \ref{HD}, 
an upper bound of the Hausdorff dimension of $\LL(\alpha)$ for the above
cases are shown.

We shall give {\bf a complete analogy of (ii) for the above two cases}. 
When $\alpha$ is an integer, 
Theorem \ref{disc_integer} gives the discrete part of the spectrum, as a small refinement %generalization 
of Dubickas \cite{Dubickas:06_2}. 
When $\alpha>3$ is a quadratic unit, the discrete part is described by 
negative continued fraction.
The results are summarized as Theorems \ref{thm:main} and
\ref{thm:main2}.
\S \ref{LBW} and \S \ref{Proofs} are devoted to their proofs. 
Our problems are
transfered into the ones in symbolic dynamics. Note that the corresponding dynamics 
is a one-sided shift when $\alpha$ is an integer, and is a two-sided shift when $\alpha$ is a
quadratic unit.  

To our surprise, in the course of the proof, we rediscover a shift space 
$\X_F$ of bi-infinite words 
over $\{0,1\}$
studied by Markoff \cite[pp.397]{Markoff:79},  defined by a set $F$ 
of forbidden subwords of the form:
$$
0w01\tilde{w}1, 1w10\tilde{w}0.
$$
Here $\tilde{w}$ is the mirror word of $w$. 
%The map intertwining the discrete spectrum of $\LL(\alpha)$ and a shift space %set of bi-infinite words 
%is quite different from the map used to 
%study the discrete part of $\L$. To our surprise, however, 
This space $\X_F$ coincides with the one %at the same shift space $\X_F$ which already 
appeared in classical Lagrange spectrum: 
the set of bi-infinite {\bf balanced words}.
%, an important object in combinatorics on words. 
%Note that the relation to the set of balanced words and Markoff spectrum is stated in \cite{Reu1}.
This coincidence (Theorem \ref{Balanced}) is
stated without proof in Cusick-Flahive \cite{Cusick-Flahive:89}: 
a written proof is given by Reutenauer \cite{Reutenauer:06} using the idea of Sturmian morphism in \cite{Lothaire:02}, 
see also Yasutomi \cite[Lemma 14]{Yasutomi}.
%Our formulation allows a one-sided version Theorem \ref{Balanced2} as well.
The discrete part of $\LL(\alpha)$ corresponds to Christoffel words, the coding of 
rotation of rational slope, and the proof is finished by using 
a certain unavoidable symmetry of sturmian words. Similar phenomena
are observed in the study of Lagrange spectrum, see \cite{Bombieri:07,Reutenauer:09}.
In summary, the problem of describing the discrete part of the classical  Lagrange spectrum 
is transfered to an optimization problem of certain values defined 
on $\X_F$. In the case of $\LL(\alpha)$, it corresponds to %For quadratic Pisot units, we found that 
a different optimization problem on $\X_F$ of similar flavor.
The map intertwining the discrete spectrum of $\LL(\alpha)$ and a shift space %set of bi-infinite words 
is quite different from the map used to study the discrete part of $\L$.

%In Section 5, we consider a shift space $\X_F$, which plays an important role on the study of $\L\cap [0,3]$. 
%For details, see \cite{bom}. 
%In particular, the set $F$ of forbidden words in Section 5 corresponds to Lemma 13 in \cite{bom}. 
%In the case where $\alpha$ is a quadratic Pisot unit, we use $\X_F$ to determine the discrete part of $\LL(\alpha)$. 

Finally we show 
in Theorem \ref{Int} that 
$\LL(\alpha)$ contains
a proper interval for a quadratic unit $\alpha>3$, which is {\bf an analogy of (iii)}. 
Our idea in brief: because it is a problem on two \textcolor{red}{sided} shift, there is a room to
recode symmetric beta expansions in base $\alpha$ to bi-infinite sequences.
Note that if $\alpha$ is an integer, then the analogy does not show because 
$\LL(\alpha)$ is of Lebesgue measure zero as in Theorem \ref{Hdim0}. 
This difference essentially comes from the nature of the corresponding symbolic dynamics whether it is two-sided
or one-sided. See the remark before Theorem \ref{Hdim0} as well.

We introduce notation on words and shift spaces which we use throughout the paper. 
%Now we outline the strategy of the proof. 
%Though  is not absolutely 
%necessary in this paper, 
%We introduce terminology of symbolic dynamics for transparency and simplicity.
Let $\B$ be a non empty finite set and $\B^*$ the free monoid generated by $\B$ with the binary operation of concatenation. %by concatenation as a binary operation. 
The empty word $\lambda$ is the identity and put $\B^+=\B^*\setminus \{\lambda\}$.
Let $\B^{\Z}$ (resp. $\B^{\N}$)  be the set of right infinite (resp. bi-infinite) words 
equipped with the product topology of discrete topologies on $\B$. %equipped with the topology of infinite Cartesian product of a discrete set $\B$.  
Then $\B^{\N}$ and $\B^{\Z}$
are compact spaces invariant under the shift map $\sigma((s_i))=(s_{i+1})$. %for $(s_i)\in \B^{\Z}$.
The topological dynamical system $(\B^{\N},\sigma)$ (resp. $(\B^{\Z},\sigma)$) 
is called one-sided (resp. two-sided) full shift.

%\section{Embedding to symbolic dynamics: $\LL(\alpha)$ is closed}
\section{Basic setup: embedding to symbolic dynamics}
\label{SD}
\subsection{Closedness of $\LL(\alpha)$ for general Pisot number $\alpha$}
A complete homogeneous symmetric polynomial in $d$ variables is defined by
\[
h_m(X_1,\ldots,X_d):=\sum_{i_1,\ldots,i_d\geq 0 \atop i_1+\cdots+i_d=m} 
X_1^{i_1}\cdots X_d^{i_d}
\]
for an integer $m\geq 0$. We can extend this definition to $m\in \Z$ through the equality:
\begin{equation}
\label{SymP}
h_m(X_1,\ldots,X_d)
=
\sum_{j=1}^d\left(
\prod_{1\leq k\leq d\atop k\ne j} \frac{X_j}{X_j-X_k}
\right)
X_j^m,
\end{equation}
whose consistency proof is found in 
\cite[Problem 7.4]{Stanley2}, \cite[Appendix A]{Louck-Biedenharn}, \cite[Theorem 1.7]{Gustafson-Milne} and \cite[Lemma 3.1]{Kaneko:09}.
We shall use the Euler identity:
$$
h_m(X_1,\ldots,X_d)=0\quad \text{for} \quad-d+1\le m\le -1.
$$
This holds because $h_m(X_1,\ldots,X_d)\prod_{i<j}(X_i-X_j)$ is an alternating polynomial of degree less than $d(d-1)/2$
(c.f. \cite{Takagi:65}, \cite[Exercises,Chapter 4]{Bourbaki_AlgebraII},\cite{Louck-Biedenharn},\cite{Kaneko:09}).

For any real number \(\xi\), there exist a unique integer 
\(u(\xi)\) and a real number \(\varepsilon(\xi)\in [-1/2,1/2)\) with 
\begin{eqnarray*}
\xi=u(\xi)+\varepsilon(\xi).
\end{eqnarray*}
Note that \(|\varepsilon(\xi)|=\|\xi\|\).
%Before we prove the closedness of $\LL(\alpha)$, 
%we introduce analogy of the right-hand sides of (\ref{Eqn:3}) and (\ref{frac}). 
%We introduce the first key of this paper: 
%the formulae in \cite{Kaneko:09} on fractional parts of geometric progressions 
%whose common ratios are Pisot numbers. 
Let $\alpha$ be a Pisot number whose minimal polynomial is $X^d+a_{d-1}X^{d-1}+\cdots+a_0\in \Z[X]$.
Denote the conjugates of $\alpha$ except itself by $\alpha_2,\ldots,\alpha_d$.
Set
\begin{align}\label{Eqn:1}
s_m=s_m(\alpha;\xi)
:=u(\xi\alpha^{m+1})+a_{d-1}u(\xi\alpha^m)+\cdots+a_0u(\xi\alpha^{m+1-d}).
\end{align}
Since
\begin{align}\label{Eqn:1b}
s_m=-\varepsilon(\xi\alpha^{m+1})-a_{d-1}\varepsilon(\xi\alpha^m)-\cdots-a_0\varepsilon(\xi\alpha^{m+1-d}),
\end{align}
we have $s_m\in \B:=[-B,B]\cap \Z$ with $B=(1+\sum_{i=0}^{d-1} |a_i|)/2$.

First we consider a geometric progression whose common ratio is an integer $\alpha=a\geq 2$. 
The sequence $(s_n)_{n\in \N}$ is bounded by $A:=(1+a)/2$ in modulus.
%There is yet another restriction that $s_{-n}(b; \xi)=0$ for sufficiently large $n$.
%Since $\tau_q=0$ for $q\le 0$, 
We have
\begin{align}
\label{Companion}
\varepsilon(\xi a^n)%=%\sum_{j=1}^{\infty}%s_{n+j}a^{-j}
=
\frac {s_{n+1}}{a} +
\frac {s_{n+2}}{a^2}+
\frac {s_{n+3}}{a^3}+
\cdots=:(s_{n+1} s_{n+2} s_{n+3}\ldots)_a. %\nonumber\\%u(\xi a^n)%&=&%\sum_{j=0}^{\infty} s_{n-j} a^j
\end{align}
Conversely, let $(s_n)_{n\in \N}$ be a bounded sequence of integers. It is easily seen for any $n\geq 0$ that 
\begin{equation}
\label{Realization}
\varepsilon(\xi a^n)
\equiv 
(s_{n+1} s_{n+2} s_{n+3}\ldots)_a
\pmod{\Z},
\end{equation}
where $\xi=(s_1 s_2 s_3\ldots)_a$. 
Problems on $(\varepsilon(\xi a^n))_{n\in \N}$ are viewed as 
the ones on the one-sided shift $(\B^{\N},\sigma)$ with $\B=[-A,A]\cap \Z$.
In \S \ref{Discrete_Int} we study sequences $(s_n)_{n\in \N}\in \B^{\N}$ that %the values 
$\limsup_{k\to \infty} |(\sigma^k((s_n)))_{a}|$ are small. 
%Conversely  for a given bounded sequence $(s_n)_{n\in \N}$,  putting
%\begin{eqnarray}
%\label{Xi0}
%\xi=\sum_{h=1}^{\infty}s_h a^{-h},
%\end{eqnarray}
%we have
%\begin{align}\label{frac0}
%\varepsilon(\xi a^n)
%\equiv
%\sum_{h=1}^{\infty}
%s_{n+h}a^{-h}
%\pmod{\Z}.
%\end{align}
%Putting $\B=([-A,A]\cap \Z)^{\N}$ and $(s_1s_2\dots)_a=\sum_{h=1}^{\infty} s_ha^{-h}$,
%the study of $\limsup_{n} |\varepsilon(\xi a^n)|$ is reduced to finding 
%$(s_n)_{n\in \N}\in \B^{\N}$ 
%which gives the infimum of $\limsup_{k\to \infty} |(\sigma^k((s_n)))_{a}|
%\in [0,1/2)$. 

Next, we consider the case that $\alpha$ is an irrational Pisot number. 
Note that 
\begin{align}\label{zerocond}
s_{-m}=s_{-m}(\alpha;\xi)=0
\end{align}
for sufficiently large $m$. 
%
%In the same way as the representation (4.5) in \cite{kan2}\footnote{We use $u(x)$ and $\varepsilon(x)$ instead of the %usual integral and fractional parts $\lfloor x\rfloor$ and $\{x\}$, respectively.} 
Our key formula in this paper is
\begin{align}\label{new1}
\varepsilon(\xi\alpha^n)
=\sum_{q=-\infty}^{\infty} s_{n+q}\tau_q,
\end{align}
where $\tau_q$ is defined by\footnote{This definition follows from an analogous formula to \cite[(4.5)]{Kaneko:09}:
$$
\varepsilon(\xi\alpha^n)
=\frac{1}{\alpha}\sum_{i=0}^{\infty}\sum_{j=0}^{\infty} \alpha^{-i}h_j(\alpha_2,\alpha_3,\ldots,\alpha_d) s_{n+i-j}.$$}
\begin{align*}
\tau_q=\begin{cases}
\displaystyle{
\sum_{j=2}^d\left(
\prod_{2\leq k\leq d\atop k\ne j} \frac{\alpha_j}{\alpha_j-\alpha_k}
\right)
\frac{\alpha_j^{-q}}{\alpha-\alpha_j}
}
&
(q\leq 0),
\vspace{1.5mm}
\\
\displaystyle{
\sum_{j=2}^d\left(
\prod_{2\leq k\leq d\atop k\ne j} \frac{\alpha_j}{\alpha_j-\alpha_k}
\right)
\frac{\alpha^{-q}}{\alpha-\alpha_j}
}
&
(q>0).
\end{cases}
\end{align*}
Conversely, for any given bounded sequence $(s_n)_{n\in \Z}$ of integers 
such that $s_{-n}=0$ for sufficiently large integer $n$, we have
\begin{align}\label{new2}
\varepsilon(\xi\alpha^n)
\equiv 
\sum_{q=-\infty}^{\infty} s_{n+q}\tau_q
\pmod{\Z}, 
\end{align}
%where $\xi$ is a suitable real number determined by $(s_n)_{n\in \Z}$. 
where 
\[
\xi=\xi((s_n)_{n\in \Z})
=\sum_{q=-\infty}^{\infty} s_{q}\tau_q.
\]
The formulae (\ref{new1}) and (\ref{new2}) are shown 
in \cite{Kaneko:09} in a more general setting.  For the convenience of the reader, 
we give a short direct proof. We observe
\begin{equation}
\label{Brief}
R_q:=\sum_{j=2}^d\left(
\prod_{2\leq k\leq d\atop k\ne j} \frac{\alpha_j}{\alpha_j-\alpha_k}
\right)
\frac{\alpha^{-q}-\alpha_j^{-q}}{\alpha-\alpha_j}=h_{-q-1}(\alpha_1,\alpha_2.\dots,\alpha_d)
\end{equation}
which follows from $h_{-1}(\alpha_1,\alpha_2,\dots,\alpha_d)=0$. Then
$R_q=0\ (q=0,\dots, d-2)$,
$R_{-1}=1$ and (\ref{Eqn:1b}) implies (\ref{new1}) and $R_q\in \Z$ for $q\le 0$ implies (\ref{new2}), which finishes 
the proof.
%}
%The double sum of this} equality converges by the assumption on $(s_n)_{n\in \Z}$. 
We write
$$
((s_n)_{n\in \Z})_{\alpha}:=\sum_{q=-\infty}^{\infty} s_{n+q}\tau_q.
$$
%Note that if $s_n=0$ for any $n\leq -1$, then \[\xi((s_n)_{n\in \Z})=((s_n)_{n\in \Z})_{\alpha}.\] 
%If $d=2$, then we get $\xi((s_n)_{n\in \Z})=(\alpha-\alpha_2)^{-1}\sum_{q=-\infty}^{\infty} s_q \alpha^{-q}$. 
Since $\alpha$ is a Pisot number, the infinite sequence $((s_n))_{\alpha}$ converges 
for any bi-infinite bounded integer sequence $(s_n)_{n\in \Z}$. We 
define a realization map $$g((s_n)_{n\in \Z})=((s_n)_{n\in\Z})_\alpha\in \R$$
and use this map even when (\ref{zerocond}) does not hold in \S \ref{QU}. Because 
$|\alpha_j|<1$ for any $2\leq j$, the value $s_{j}$ for $j\le n_0$ has no influence to the limsup value
for any fixed $n_0$,
we embed our problem on $\limsup_{n\to \infty} \|\xi \alpha^n\|$ to the one on symbolic 
dynamics $(\B^{\Z},\sigma)$  in light of (\ref{new1}) and (\ref{new2}).
% in question for any fixed $n$.
%into the one of 
%the two-sided shift $(\B^{\Z},\sigma)$.  %Therefore the study of $\limsup_{n} |\varepsilon(\xi\alpha^n)|$ is reduced to the symbolic counter part to find $(s_n)\in \B^{\Z}$ which gives the infimum of $\limsup_{k\to \infty} |(\sigma^k((s_n)))_{\alpha}|\in [0,1/2)$.

In the case where $\alpha=a$ is an integer, putting 
\[
\tau_q:=
\begin{cases}
0 & (q\leq 0)\\
a^{-q} & (q>0),
\end{cases}
\]
we see (\ref{new1}) and (\ref{new2}) are valid, which gives a unified proof for the following theorem. 

\begin{theorem}
\label{Closed}
Let $\alpha$ be a Pisot number. Then 
$\LL(\alpha)$ %}:=\left\{\left. \limsup_{n\to \infty} \|\xi \alpha^n\|\ \right|\ \xi\in \R \right\}$
is a closed subset of $[0,1/2]$.
\end{theorem}

%Note for the reader: the following proof may look intricate at first sight
%because (\ref{Zero}) is not a closed condition, 
%i.e., the property (\ref{Zero}) no longer holds by our limiting process. 

\begin{proof}
%We first prove that $\LL(\alpha)$ is closed. 
%It does not cause an essential difference whether the symbolic space is one-sided or two-sided.
%Hereafter we describe it when $\alpha$ is a quadratic unit. 
%Take a sequence $(\eta_j)$ converging to $\eta$ and its associated $(\xi_j)$ such that
%$\eta_j= \limsup_{n\to \infty} \|\xi_j \alpha^n\|$.
Let $\eta$ be a limit of the sequence $(\eta_j)_{j\in \N}$ with $\eta_j= \limsup_{n\to \infty} \|\xi_j \alpha^n\|$. 
We may assume that $\eta_j\neq \eta$ for all $j$. 
%As we discussed above, it is no harm to assume that $s_{n}=0$ for $n<0$. 
%By (\ref{Xi}), we have $$|\xi_j|\le \frac {B}{(1-\alpha^{-2})(1-\alpha^{-1})}.$$  
%Thus there exists
%a subsequence of $(\eta_j)$ so that its corresponding $(\xi_j)$ converges in $\R$.
For each $\xi_j$, the sequence $(s[j]_n)_{n\in \Z}\in \B^{\Z}$ is defined by (\ref{Eqn:1}). 
For a positive integer $K$, define
$$
W_K=\{ \sigma^k ((s[j]_n))\ |\ k,j \ge K \}  
$$
and the topological upper limit \cite[p.25]{Kechris:95}:
$$
W^+=\bigcap_{K=1}^{\infty} \overline{W_K},
$$
where $\overline{V}$ is the closure of $V$ with respect to the topology of $\B^{\Z}$.
We are interested in the set
$$
W=\{ (s_n)\in \B^{\Z}\ |\ |((s_n))_\alpha|=\eta, \text{ and } (s_n)\in W^+ \}.
$$
By definition $W$ is compact. Each element in $W$ is the limit of a sequence of a certain shift of 
subsequence of $(s[j]_n)$ having its $g$-value $\pm \eta$. 

We claim that $W$ is non empty, and moreover 
one can select $(s_n)\in W$ so that there exists a sequence $(j_m)_{m\in \N}$ 
of positive integers that every central block $s_{-\ell}\dots s_{\ell}$
appears infinitely often in $(s[j_p]_n)_{n\in \Z}$ for $p\ge \ell$.
Hereafter a sequence $(s_n)\in W$ having this property is called {\bf limit recurrent}.

We now show the claim above. Put $\varepsilon_j=|2(\eta-\eta_j)|$.
By $\eta-\varepsilon_{j}\le \eta_{j}-\varepsilon_{j}/2$ and 
$\limsup_{k\to \infty} |g(\sigma^{k}((s[j]_n)))|
=\eta_{j}$, the
set
$$
A(j):=
\{ k\in \N\ |\ |\ g(\sigma^{k}((s[j]_n)))|>\eta-\varepsilon_{j} \} 
$$
is infinite. 
Let $m(0,j)=j$ and we shall define $s_{-\ell}\dots s_{\ell}\in \B^{2\ell+1}$ and 
a sequence $(m(\ell,j))_{j\in \N}$ for $\ell\in \N$ so that 
$(m(\ell+1,j))_j$ is a subsequence of $(m(\ell,j))_j$.
%, and consequently 
%$m(k,j)\ge j$ and $\varepsilon_{j+1}<\varepsilon_j$ hold.
Assume that $(m(\ell,j))_j$ is already defined and the set
$$
\left\{n\in A(m(\ell,j))\ 
\left|\ s[m(\ell,j)]_{n-\ell+1}\dots s[m(\ell,j)]_{n+\ell-1}=s_{-\ell+1}\dots s_{\ell-1} 
\right.\right\}$$
is infinite\footnote{
This condition is trivially valid when $\ell=0$.} for all $m(\ell,j)$.
We choose a word $s_{-\ell}^{m(\ell,j)}\dots s_{\ell}^{m(\ell,j)}\in \B^{2\ell+1}$ satisfying 
$s_{-\ell+1}^{m(\ell,j)}\dots s_{\ell-1}^{m(\ell,j)}=s_{-\ell+1}\dots s_{\ell-1}$
so that
$$
\left\{ n\in A(m(\ell,j))\ \left|\ s[m(\ell,j)]_{n-\ell}\dots s[m(\ell,j)]_{n+\ell}=s_{-\ell}^{m(\ell,j)}\dots s_{\ell}^{m(\ell,j)} \right.\right\}
$$
is infinite.
Then we select a subsequence $(m(\ell+1,j))_j$ of $(m(\ell,j))_j$ and $s_{-\ell}\dots s_{\ell}\in \B^{2\ell+1}$
with the property that 
for every $j$ the set
$$
\left\{n\in A(m(\ell+1,j))\ 
\left|\ s[m(\ell+1,j)]_{n-\ell}\dots s[m(\ell+1,j)]_{n+\ell}=s_{-\ell}\dots s_\ell \right.\right\}
$$
is infinite. We continue this selection inductively on $\ell$. Then the sequence $(s_{n})_{n\in \Z}$
should satisfy $((s_n))_{\alpha}=\eta$ and every central block $s_{-\ell}\dots s_\ell$ appears infinitely
often in $(s[m(p,j)]_n)$ for $p>\ell$, which implies the claim.

Let us fix a limit recurrent $(s_n)\in W$. If 
$\limsup_{k\to \infty} |\sigma^k((s_n))_{\alpha}|=\eta$, defining
$$
s'_n=\begin{cases} s_n & n\ge 0 \\
                   0  & n<0\end{cases}
$$
we obtain $\eta=\limsup_{n \to \infty} \|\xi \alpha^n\|$ for
 $\xi=((s'_n))_\alpha$.
The proof becomes tricky when 
$
\limsup_{k\to \infty} |\sigma^k((s_n))_{\alpha}|<\eta.
$
Let $t(\ell)=s_{-\ell}\dots s_{\ell}$.
% with $\ell\ge m+ N$ and write $t(\ell)=u_{\ell}t_{\ell}v_{\ell}$
%with $u_{\ell}, v_{\ell}\in \Xi$ with $|u_{\ell}|=|v_{\ell}|=N$. 
Since $(s_n)$ is limit recurrent, $t(\ell)$ appears infinitely often in $(s[m(p,j)]_n)$
for $p>\ell$. Thus there exists $a_{\ell} \in \B^*$
% with $|a_{\ell}|\ge N$
that $t(\ell) a_{\ell} t(\ell)$ is a factor of $(s[m(p,j)]_n)$. 
We construct a word $(x_n)$ that $x_n=0$ for $n<0$ and
$$
x_{0}x_{1}\dots = t(\ell)
s_{\ell+1} a_{\ell+1} t(\ell+1)
s_{\ell+2} a_{\ell+2} t(\ell+2)s_{\ell+3} a_{\ell+3} 
\dots
$$
Since $\lim_{\ell\to\infty} |(t(\ell))_\alpha|=\eta$,
%and $$\mathrm{Fac}_N(v_{\ell+k}a_{\ell+k+1}u_{\ell+k+1})\subset \Xi,$$ 
we have $\limsup_{k\to \infty} |g(\sigma^k((x_n)))|\ge \eta$.
By construction,  the subword 
$x_i\dots x_{i+\ell+k}$ for $i \ge \sum_{j=0}^{k-1} |t(\ell+j)s_{\ell+j+1}a_{\ell+j+1}|$ is
a factor of $(s[m(p,j)]_n)$ for $p>\ell+k$. 
Since we may choose arbitrary large $k$, this implies that 
$\limsup_{k\to \infty} |g(\sigma^k((x_n)))|\le \eta$, finishing the proof.

\end{proof}
%\begin{rem}
%\begin{rm}
%Theorem \ref{Closed} holds for any Pisot numbers $\alpha$. The proof relies on general formulae for $\varepsilon(\xi\alpha^n)$ (see Propositions 4.1 and  5.2 in \cite{kan2}). 
%\end{rm}
%\end{rem}

%Let us deduce an easy case that $\alpha=a\ge 2$ is an integer. 
%For each real \( \xi \) and integer \(n\in \N\), we have
%\begin{eqnarray*}
%s_n=s_n(a; \xi )=u( \xi  a^{n})-a u( \xi  a^{n-1})=-\varepsilon( \xi a^{n})+a \varepsilon( \xi  a^{n-1})
%\end{eqnarray*}
%We have 
%\begin{eqnarray}
%s_n(\alpha; \xi )
%\label{Eqn:0}
%\end{eqnarray}

\begin{rem}
\label{CusickAlt}
\begin{rm}
The method for Theorem \ref{Closed} gives an alternative proof that the classical Lagrange spectrum $\L$ is closed. 
In fact, set
\[[a_0;a_1,a_2,a_3\ldots]:=a_0+\cfrac{1}{a_1+\cfrac {1}{a_2+ \cfrac{1}{a_3+ \cfrac {1}{\ddots} } }}.\]
For a bi-infinite sequence $\ba=(a_n)_{n\in \Z}$ of positive integers, define
\[
\varphi(\ba):=\limsup_{n\to\infty}(
[a_n;a_{n+1},a_{n+2},\ldots]+[0;a_{n-1},a_{n-2},\ldots]
).
\]
It is well known (see \cite{Cusick-Flahive:89,Bombieri:07}) that 
\[\L=\bigcup_{A=1}^{\infty}\{
\varphi(\ba)\mid \ba=(a_n)_{n\in \Z}, 1\leq a_n\leq A \mbox{ for all }n
\}=:\bigcup_{A=1}^{\infty} \L_A.\]
In the same way as the proof of Theorem \ref{Closed}, we see that $\L_A$ is closed for any $A\geq 1$
and thus $\L$ is closed, because, for any positive integer $A$, we have $\L\cap [0,A]=\L_A \cap [0,A]$. %it is a locally finite union of closed sets. 
\textcolor{red}{Indeed the closedness proofs of $\LL(\alpha)$ and $\L$ have some part in common. 
The key idea in Cusick \cite{Cusick:75} is to show that each element of $\L$ is well 
approximated by purely periodic sequences. One can confirm that such purely periodic sequences are easily produced by limit recurrent words as well.}
\end{rm}
\end{rem}

%For the case of an integer $a$, we use the fact that $((s_n))_a$ is a one-sided series. 
%On the other hand, if  $\alpha$ %$\alpha=(b+\sqrt{b^2\mp 4})/2$ 
%is a quadratic Pisot unit, then 
%$((s_n))_\alpha$ is a two-sided series, which is the reason of the difference from the integer case. 
%Compare Theorem \ref{Hdim0} with Theorems \ref{Hdim} and \ref{Int}.

\subsection{Upper bounds for Hausdorff dimension}
\label{HD}
We write down a concrete form of $((s_n))_\alpha$ in the case when $\alpha=(b+\sqrt{b^2\mp 4})/2$ is a quadratic Pisot unit.
For each real \( \xi \) and integer \(n\), we have 
\begin{align}\label{Eqn:2}
s_n=u( \xi  \alpha^{n+1})-b u( \xi  \alpha^{n})
\pm u( \xi \alpha^{n-1})=-\varepsilon( \xi  \alpha^{n+1})+b \varepsilon( \xi  \alpha^{n})
\mp \varepsilon( \xi \alpha^{n-1})
\end{align}
and $|s_n|\leq B$, where $B=(2+b)/2$. 
Then we see that $((s_n))_\alpha$ is represented as 
\begin{align}
\label{Eqn:3}
\frac{1}{\alpha-\alpha_2}
\left( \dots +s_{n-2}\alpha_2^2+s_{n-1} \alpha_2 + s_{n}+ \frac {s_{n+1}}{\alpha} +
\frac {s_{n+2}}{\alpha^2}+\dots \right).%\nonumber\\
\end{align}

%Denote by $\B^{\Z}$ the the set of bi-infinite words over $\B$ endowed with the topology 
%of Cartesian product and 
%Define
%\[\LL(\alpha):=\left\{\left. \limsup_{n\to \infty} \|\xi \alpha^n\|\ \right|\ \xi\in \R \right\}.\]

%The relation (\ref{Eqn:3}) and (\ref{frac}) holds for more general Pisot number of degree $d\geq 2$. 

%In what follows, we consider the case where 
%$\alpha$ be an integer $a\ge 2$ or a quadratic Pisot unit $(b+\sqrt{b^2\mp 4})/2$. 
If $\alpha$ is an integer (resp. a quadratic Pisot unit), then $(s_n)$ represents $(s_n)_{n\in \N}\in \B^{\N}$ 
(resp. $(s_n)_{n\in \Z}\in \B^{\Z}$). 
For a given $\xi\in \R$, let $\eta=\limsup_{n\to \infty} \|\xi \alpha^n\|$ and $(s_n)$ be 
an infinite sequence corresponding to $\xi$. 
Then we find
a subsequence $0\le n_1<n_2<\dots$ that $\eta=\lim_{j\to \infty} \|\xi \alpha^{n_j}\|$.
By a similar but simpler discussion as in the above proof, 
in the sequence $(\sigma^{n_j}((s_n)))_{j=1,2,\dots}$ we can
select a converging subsequence and obtain a limit $(w_n)$ by the 
topology of $\B^{\N}$ or $\B^{\Z}$.
By construction, we have
\begin{equation}
\label{Limsup}
|((w_n))_{\alpha}|=\eta,\qquad \big|\big(\sigma^k ((w_n))\big)_{\alpha}\big|\le \eta \text{ for all } k,%\in \N \text{ or } \Z.
\end{equation}
where if $\alpha$ is an integer (resp. a quadratic Pisot unit), then $k$ moves through positive integers (resp. integers). 
We call this word $(w_n)\in \B^{\N} \cup \B^{\Z}$ a {\bf limsup word} for $\xi$.
\begin{rem}\label{rem:limsup}
A limsup word $(w_n)$ is  not generally uniquely determined by $(s_n)$. Each finite subword of any limsup word  appears infinitely many times in $(s_n)$.
This makes our later discussion significantly simpler.
\end{rem}
We give several applications of limsup words in the sequel.
Let us denote the Hausdorff dimension of $Z\subset \R$ by $\mathrm{dim}_H(Z)$.
Hereafter we study $\mathrm{dim}_H \left(\LL(a)\cap [0,t]\right)$ when 
$\alpha$ be an integer $a\ge 2$ or a quadratic Pisot unit $(b+\sqrt{b^2\mp 4})/2>3$. 

We give an important remark : the limsup word is two-sided when $\alpha$ is a quadratic unit, and it 
does not satisfy (\ref{zerocond}). This means we can not apply (\ref{new2}).
When $\alpha$ is an integer, the limsup word is one-sided and this problem does not happen (see (\ref{Realization})).
This difference makes Theorem \ref{Hdim0} sharper than Theorem \ref{Hdim}.
We will see later that the difference is really large in Theorem \ref{Int}: 
for a quadratic unit $\alpha>3$, $\LL(\alpha)$ contains a proper interval.
First we consider the case of $\alpha=a\in \Z$.

\begin{theorem}
\label{Hdim0}
%For any integer $a\geq 2$, we have 
$$
\mathrm{dim}_H \left(\LL(a)\cap [0,t]\right)\leq\log(1/2)/\log(t)
$$
for any integer $a\geq 2$. In particular, $\LL(a)$ has Lebesgue measure zero.
%is a continuous non-increasing function from $[0,1/2]$ to $[0,1]$ that 
%In particular, if $t<1/2$ then $ \mathrm{dim}_H \LL(a)\cap [0,t]$ 
\end{theorem}

\begin{proof}
Let $t\in [0,1/2)$ and $\eta\in \LL(a)\cap [0,t]$. Choose a limsup word $(w_n)\in \B^{\N}$
for $\eta$. Then 
$$
(w_{n+1}w_{n+2}\dots)_a \in [-t,t],
$$
i.e., $\{ \eta a^n \} \in [0,t]\cup [1-t,1]$ holds for any $n\in \N$  by (\ref{Limsup}). %, see the above remark.
Using appropriate coverings of $[0,t]$ and $[1-t,1]$
by cylinder sets of the $a$-adic expansion map $x\mapsto ax \pmod{1}$ on $[0,1]$, 
we see that 
$\eta$ belongs to the attractor $Z$ of the iterated function system
$$
Z=t Z \cup (t Z+1-t),
$$
where $\mathrm{dim}_H(Z)=\log(1/2)/\log(t)$. %whose Hausdorff dimension is $\log(1/2)/\log(t)$. 
The latter statement follows from
$$
\LL(a)-\{1/2\}=\bigcup_{n=1}^{\infty} \left(\LL(a)\cap \left[0,\frac 12- \frac 1n\right] \right).
$$
%$$
%\mathrm{dim}_H 
%\left.\left\{ \eta \in [0,t] \ \right|\ \{ \eta a^n \} \in [0,t]\cup [1-t,1]\ \forall n\in \N\ \right\}.
%$$
%we obtain the desired upper bound of the dimension.
\end{proof}

%\begin{cor}
%\label{LebZero0}
%The set $\LL(a)$ has Lebesgue measure zero.
%\end{cor}

%\begin{proof}
%This follows from
%$$
%\LL(a)-\{1/2\}=\bigcup_{n=1}^{\infty} \left(\LL(a)\cap \left[0,\frac 12- \frac 1n\right] \right).
%$$\end{proof}

%When $\alpha$ is a quadratic Pisot unit $(b+\sqrt{b\mp 4})/2$, we do not have 
%an associated interval map. However, using our method, 
%we can obtain
Next, we consider the case where $\alpha$ is a quadratic Pisot unit \textcolor{red}{$(b+\sqrt{b^2\mp 4})/2$.} 

\begin{theorem}
\label{Hdim}
We have $$
\mathrm{dim}_H \left(\LL(\alpha)\cap [0,t]\right)\le \frac {\log(4m+1)}{\log(\alpha)}$$
for $m=m(t)=\lfloor(2+b)t \rfloor$. In particular, $\LL(\alpha)\cap [0,t_0]$ has Lebesgue measure zero, 
where $t_0=(\alpha-1)/(4(b+2))(<1/4)$. 
\end{theorem}

\begin{proof}
For $k \in \N$, let 
$Z_{k}=\left\{ \sum_{i=1}^{\infty} a_i \alpha^{-i}\ |\ a_i\in [-k,k]\cap \Z \right\}$.
It is well known that $\mathrm{dim}_H Z_{k}=\min \{1, \log (2k+1)/\log \alpha\}$, see e.g.
\cite[Example 4.5]{Falconer:90}.
If \textcolor{red}{$\limsup_{n\to \infty} \| \xi \alpha^n\|< t$,} then 
from (\ref{Eqn:2}) 
we have
$$
|s_n(\alpha; \xi)|\le (2+b) t
$$
for sufficiently large $n$. 
By considering limsup words, this shows that $$
\textcolor{red}{\LL(\alpha)\cap [0,t)} \subset \textcolor{red}{\frac{1}{\alpha-\alpha_2}}\bigcup_{a_0=-m}^{m}
\left(
a_0+\alpha^{-1} Z_{2m}\right).
$$
%where $|a_0|\le B$ and $|a_i|\le 2B$. 
Since Hausdorff dimension is preserved by finite or countable union, we obtain the former part. 
For any real number $0\leq t<t_0$, we get by the irrationality of $\alpha$ that 
$$
m(t)\leq %\lfloor t(b+2)\rfloor \le 
\left\lfloor \frac{\alpha-1}{4}\right\rfloor<\frac{\alpha-1}{4},
$$
which implies the latter part. 
\end{proof}

%\begin{cor}
%\label{LebZero}
%The set $\LL(\alpha)\cap \left[0, \frac {\alpha-1}{4(b+2)}\right]$ has Lebesgue measure zero.
%\end{cor}

%\begin{proof}
%If $4\lfloor t(b+2)\rfloor +1 <\alpha$, then $\mathrm{dim}_H(\LL(\alpha)\cap [0,t])<1$.
%Since $\alpha$ is irrational, this is equivalent to 
%$$
%\lfloor t(b+2)\rfloor \le \left\lfloor \frac{\alpha-1}{4}\right\rfloor.
%$$
%\end{proof}

%Note that $t_0<1/4$. %$(\alpha-1)/(4(b+2))<1/4$. 
Theorem \ref{Hdim} suggests that smaller the value $t$ we obtain a more sparse 
set $\LL(\alpha)\cap [0,t]$. 
We go further to the case when this set becomes discrete.
Our goal is to show that the limsup word has a lot of forbidden 
factors, and such a word does not exist, or exists but in a very special form 
if $t$ is sufficiently small.

\section{Discrete part of $\LL(\alpha)$: integer case}
\label{Discrete_Int}
%Let $\alpha=a\in\Z_{\geq 2}$ and put
%\[\mathcal{L}(a):=\left\{\left.\limsup_{n\to\infty} \|\xi b^n\| \ \right| \ \xi\in \R\right\}.\]
%\[
%{{E}}(X)=\frac{1-(1-X)\prod_{m=0}^{\infty} (1-X^{2^m})}{2X}
%\]
Recall that $E(X)$ is defined by (\ref{inf_prod}).
For $k\geq 0$, we put 
\[
{{E}}^{(k)}(X):=\frac{1+X^{2^k}-(1-X)\prod_{m=0}^{k-1} (1-X^{2^m})}{2X(1+X^{2^k})},
\]
where $\prod_{m=0}^{k-1} (1-X^{2^m})=1$ if $k=0$. Note for any integer $a\geq 2$ that 
\[\lim_{k\to\infty}{{E}}^{(k)}\left(\frac{1}{a}\right)={{E}}\left(\frac{1}{a}\right).\]
Put 
\[\Xi(a):=\mathcal{L}(a)\cap\left(0,\frac{1}{a}{{E}}\left(\frac{1}{a}\right)\right].\]

\begin{theorem}\label{disc_integer}
For any integer $a\geq 2$, we have 
\[\Xi(a)=\left\{\left.\frac{1}{a}{{E}}^{(k)}\left(\frac{1}{a}\right) \ \right| \ k=0,1,2,\ldots\right\}\bigcup
\left\{ \frac{1}{a}{{E}}\left(\frac{1}{a}\right) \right\}.\]
\end{theorem}

\begin{proof}
Let $\eta\in \Xi(a)$. There exists
\(
\xi=\sum_{n=1}^{\infty}s_n a^{-n},%\frac{s_n}{a^n},
\)
where $s_n=s_n(a;\xi)$, such that $\eta=\limsup_{n\to\infty}\|\xi a^n\|$. Let $\bt:=(t_n)_{n\geq 1}$ be a limsup word of 
$\xi$. 
%Let $t_n:=x_n$ if \[\sum_{n=1}^{\infty}\frac{x_n}{a^n}>0,\] and 
%$t_n:=-x_n$ otherwise. Then we see 
If necessary, changing $\xi$ and $(t_n)_{n\geq 1}$ by $-\xi$ and $(-t_n)_{n\geq 1}$, we have 
\(
\eta=\sum_{n=1}^{\infty}t_n a^{-n}.%\frac{t_n}{a^n}.
\)
%Let $\sigma$ be the one-sided shift. 
From the proof of Theorem 4 in Dubickas \cite{Dubickas:06_2}, we see that 
$\bt\in\{\overline{1},0,1\}^{\infty}$ and the following factors (subwords) are forbidden: 
\begin{align*}
\begin{cases}
10^k1, \overline{1}0^k\overline{1} & k\geq 0,\\
10^m\overline{1}, \overline{1}0^m 1 & m\geq 2,
\end{cases}
\end{align*}
where $\overline{1}$ denotes $-1$. Note that each digit from $\{1,-1\}$ appears infinitely often in $\bt$. 
%If necessary, changing $\bs$ by $\sigma^N(\bs)$ for a suitable $N\in \N$, we may assume that $\bs$ has the form 
%\[
%\bs=1 0^{y_1-1} \overline{1} 0^{y_2-1} 1 0^{y_3-1} \overline{1} \ldots,
%\]
%where $\by=(y_i)_{i\geq 1}\in \{1,2\}^{\infty}$. 
Therefore $\bt=(t_n)_{n\geq 1}$ has the form 
\[
\bt=1 0^{v_1-1} \overline{1} 0^{v_2-1} 1 0^{v_3-1} \overline{1} \ldots=:\Phi(\bv),
\]
where $\bv=(v_i)_{i\geq 1}\in \{1,2\}^{\infty}$. 
%Note that $\bv$ is a limit point of $(\sigma^n(\by))_{n\geq 1}$. 
Let $\mathcal{F}$ be the set of limsup words $\bt$ with initial letter 1 corresponding to a certain $\eta\in \Xi(a)$.
%We denote the details of $\mathcal{F}$. 
For any $\by=(y_i)_{i\geq 1}\in \{1,2\}^{\infty}$, let 
\[
f(\by;X)=f(y_1y_2\ldots;X):=1-X^{y_1}+X^{y_1+y_2}-+\cdots=\sum_{i=0}^{\infty}(-1)^i X^{y_1+\cdots+y_i}.
\]
Set 
\[
\mathcal{W}:=\left\{\by\in\{1,2\}^{\infty} \ \left| \ 
\limsup_{n\to\infty}\textcolor{red}{\left\|\frac1a f\left(\by;\frac1a\right)\cdot a^n\right\|}\in \Xi(a)\right.\right\}.
\]
Let $\by=y_1y_2\ldots, \by'=y_1'y_2'\ldots$ be finite or infinite words on the alphabet $\{1,2\}$. We denote $\by>\by'$ 
if there exists $n\geq 1$ satisfying $y_n\ne y_n'$ and 
\[\textcolor{red}{(-1)^{h+1}(y_h-y_h')>0},\]
where $h=\min \{n\geq 1 \mid y_n\ne y_n'\}$. We denote by $\by\geq \by'$ if $\by>\by'$ or $\by=\by'$. 
Assume that $\by,\by'\in\{1,2\}^{\infty}$ satisfy $\by\geq \by'$. Then 
\(f(\by;a^{-1})\geq f(\by';a^{-1}),\)
where the equality holds only if $\by=\by'$. 
Hence, we obtain 
\[
\mathcal{F}=\left\{\left.
\Phi\left(\limsup_{N\to\infty} \sigma^N(\by)\right)
\ \right| \ \by\in \mathcal{W}
\right\},
\]
where $\limsup_{N\to\infty} \sigma^N(\by)$ is determined by the order $\geq$ defined above.

Let $\tau$ be the substitution on $\{1,2\}^{*}$ defined by $\tau(1)=2$ and $\tau(2)=211$. 
Let $A_n$ ($n=0,1,\ldots$) be the finite words defined by 
$A_0=1$ and $A_n=\tau(A_{n-1})$ for $n\geq 1$. 
Then 
\[\bw=w_1w_2\ldots:=\lim_{n\to\infty}A_n\]
is a fixed point of $\tau$. 
Recall that 
\[\frac{1}{a}{{E}}\left(\frac1a\right)=\frac1a f\left(\bw;\frac1a\right)\]
(see relation (9) in \cite{Dubickas:06_2}).
Hence, $\mathcal{W}$ is the set of $\by=(y_i)_{i\geq 1}\in \{1,2\}^{\infty}$ satisfying the following: 
For any $u_1\ldots u_l\in \{1,2\}^{+}$ with $u_1\ldots u_l>\bw$, the word $u_1\ldots u_l$ appears 
at most finitely many times in $\by$.

Let us determine all eventually periodic words in $\mathcal{F}$. 
Let $\mathcal{F}_k$ ($k=2,3,\ldots$) be the subsets of $\mathcal{F}$ defined by %$\mathcal{F}_1:=\mathcal{F}$, 
$\mathcal{F}_2:=\mathcal{F}\backslash\{A_0^{\infty},A_1^{\infty}\}$, 
and $\mathcal{F}_{k+1}:=\mathcal{F}_k\backslash\{A_k^{\infty}\}$ for $k\geq 2$, where 
$A_k^{\infty}=A_kA_kA_k\ldots\in\{1,2\}^{\infty}$. 
Let $\bv=\limsup_{N\to\infty}\sigma^N (\by)\in\mathcal{F}$ with $\by\in\mathcal{W}$.
Refining the proof of Lemma 4 in \cite{Dubickas:06_2}, 
we show the following claim by induction: Let $k\geq 2$. Then, if $\bv\in \mathcal{F}_k$, then the word $A_k$ 
is a prefix of $\bv$, and $\bv$ is a concatenation of $A_{k-1}$, $A_k$. 
Moreover, $\by$ can be taken as a word which is a concatenation of $A_{k-1}$, $A_k$. 
In particular, the claim above implies that 
$\cap_{k=2}^{\infty}\mathcal{F}_k=\{\bw\}$, i.e., $\bw$ is a unique non-eventually periodic word in $\mathcal{F}$.

If $y_n=1$ holds for sufficiently large $n$, then $\bv=1^{\infty}=A_0^{\infty}$. 
Similarly, if $y_n=2$ for sufficiently large $n$, then $\bv=2^{\infty}=A_0^{\infty}$. 
We now prove the claim in the case of $k=2$. 
Let $\bv\in \mathcal{F}_2$. Then each digit from $\{1,2\}$ appears infinitely often in $\by$. 
Observe that $212>\bw$ and $2111>\bw$. 
Thus, we may assume that neither $212$ nor $2111$ does not appear in $\by$. 
Hence, $\by$ is a concatenation of $A_1=2$ and $A_2=211$. 
Consequently, $\bv$ is also a concatenation of $A_1$ and $A_2$ because 
the initial letter of $\limsup_{N\to\infty}\sigma^N(\by)$ is 2. 
Since $211>22$, we see that $A_2$ is a prefix of $\bv$. 

Suppose that the claim is true for $k=\ell\geq 2$. Let $\bv\in \mathcal{F}_{k+1}$. 
By the induction hypothesis, we may assume that $\by$ has the form 
$\by=\tau^{k-1}(z_1z_2\ldots)$ with $\bz=(z_i)_{i\geq 1}\in \{1,2\}^{\infty}$
Since $\bv\ne A_{k-1}^{\infty},A_k^{\infty}$, each digit from $\{1,2\}$ appears infinitely often in $\bz$. 
Recall that $\tau^{k-1}(212)>\bw$ and $\tau^{k-1}(2111)>\bw$. 
Thus, we may assume that neither $212$ nor $2111$ does not appear in $\bz$. 
Hence, $\bz$ is a concatenation of $2=\tau(1)$ and $211=\tau(2)$. 
Consequently, $\by$ and $\bv$ are concatenations of $\tau^{k-1}(2)=\tau^{k}(1)=A_{k}$ and 
$\tau^{k-1}(211)=\tau^{k}(2)=A_{k+1}$. 
Note by $\bv\in \mathcal{F}_{k+1}\subset\mathcal{F}_k$ that  $A_k=\tau^{k}(1)$ is a prefix of $\bv$. 
Observe $\bv=\limsup_{N\to\infty}\sigma^N(\tau^{k}(\bz'))$ with certain $\bz'\in \{1,2\}^{\infty}$ and that $2$ appears 
infinitely often in $\bz'$ since $\bv\ne \tau^{k}(1)^{\infty}$. 
We see that either $\tau^{k}(2)$ or $\tau^{k}(11)$ is a prefix of $\bv$. %\footnote{Strictly speaking, we used the following fact: Let $U$ be a sufficiently long finite subword of $\tau^{k}(\bz')$ with a prefix $\tau^{k}(1)$. Then either $\tau^{k}(2)$ or $\tau^{k}(11)$ is a prefix of  $U$. This fact can be shown by induction on $k$ because $\tau^{k}(11)$ is a prefix of $\tau^{k}(12)$ by $\tau^{k}(2)=\tau^{k-1}(211)=\tau^{k}(1)\tau^{k-1}(1)\tau^{k-1}(1)$.}. 
Since $\tau^{k}(2)>\tau^{k}(11)$, we get that $\tau^{k}(2)=A_{k+1}$ is a prefix of $\bv$. Therefore, we proved the claim.

Hence, we obtain 
\[
\mathcal{F}=\{A_k^{\infty}\mid k=0,1,2,\ldots\}\cup \{\bw\}
\]
and 
\[\Xi(a)=\left\{\left.
\frac1a f\left(A_k^{\infty};\frac1a\right)%\frac{1}{a}{{E}}^{(k)}\left(\frac{1}{a}\right) 
\ \right| \ k=0,1,2,\ldots\right\}\bigcup
\left\{ \frac{1}{a}{{E}}\left(\frac{1}{a}\right) \right\}.\]
In what follows, we prove for any $k\geq 0$ that 
\begin{align}\label{eqn:disc_integer1}
f\left(A_k^{\infty};\frac1a\right)={{E}}^{(k)}\left(\frac{1}{a}\right),
\end{align}
which implies Theorem \ref{disc_integer}. For any $k\geq 0$, let 
$A_k=:a_1 a_2 \ldots a_{\ell(k)},$
where $a_i\in\{1,2\}$ for any $i$ and $\ell(k)$ is the length of $A_k$. Set  
$s_j:= \sum_{i=1}^j a_i $
for $1\leq j\leq \ell(k)$. 
It is easily seen by induction that $\ell(k)$ is an odd number and $s_{\ell(k)}=2^k$. 
Observe that 
\begin{align}
f(A_k^{\infty};X)
&=
1+\sum_{m=0}^{\infty}(-X^{s_{\ell(k)}})^m\cdot 
(-X^{s_1}+X^{s_2}-X^{s_3}+-\cdots-X^{s_{\ell(k)}})\nonumber\\
%&=
%1+\frac{1}{1+X^{s_{\ell(k)}}}\cdot 
%(-X^{s_1}+X^{s_2}-X^{s_3}+-\cdots-X^{s_{\ell(k)}})\nonumber\\
%&=\frac{1}{1+X^{2^k}}\cdot 
%(1-X^{s_1}+X^{s_2}-X^{s_3}+-\cdots-X^{s_{-1+\ell(k)}})\nonumber\\
&=:\frac{1}{1+X^{2^k}}\cdot P_k(X)\label{eqn:disc_integer2},
\end{align}
where $P_k(X)=1-X^{s_1}+X^{s_2}-X^{s_3}+-\cdots-X^{s_{-1+\ell(k)}}$. 
From $2^{k+1}\geq 1+2^k$, we get
\begin{align*}
&1-(1-X)\prod_{m=0}^{\infty} (1-X^{2^m})\\
%&=1-(1-X)\prod_{m=0}^{k} (1-X^{2^m})+O\left(X^{1+2^k}\right)\\
&=1+X^{2^k}-(1-X)\prod_{m=0}^{k-1} (1-X^{2^m})+O\left(X^{1+2^k}\right),
\end{align*}
where $O(X^n)$ denotes the terms of degree not less than $n$. 
Moreover, using 
\[
P_k(X)={{E}}(X)+O\left(X^{2^k}\right),
\]
we see 
\begin{align*}
P_k(X)&=
\frac{1-(1-X)\prod_{m=0}^{\infty} (1-X^{2^m})}{2X}+O\left(X^{2^k}\right)
\\
&=
\frac{1+X^{2^k}-(1-X)\prod_{m=0}^{k-1} (1-X^{2^m})}{2X}+O\left(X^{2^k}\right).
\end{align*}
Since the degree of $P_k(X)$ is less than $2^k$, we obtain 
\[P_k(X)=
\frac{1+X^{2^k}-(1-X)\prod_{m=0}^{k-1} (1-X^{2^m})}{2X}.\]
Substituting the equality above into (\ref{eqn:disc_integer2}), we deduce (\ref{eqn:disc_integer1}). 
\end{proof}

\section{Discrete part of $\LL(\alpha)$: quadratic unit case}
\label{QU}

Let $\alpha>1$ be a quadratic unit, i.e., a root of $x^2-bx\pm 1$ with $b\in \N$
and $\alpha_2=\pm 1/\alpha$ be its conjugate. 
We have shown that $\LL(\alpha)$ %$\{ \limsup_n \|\xi \alpha^n\| \ |\ \xi\in \R\}$ 
is a closed subset of $[0,1/2]$
in Theorem \ref{Closed} and that $\LL(\alpha)\cap [0,(\alpha-1)/(4(b+2))]$
%$$
%\left\{ \limsup_n \|\xi \alpha^n\| \ |\ \xi\in \R \right\} \cap \left[0, \frac {\alpha-1}{4(b+2)}\right] 
%$$
has Lebesgue measure zero in Theorem \ref{Hdim}.
In this section, we give the minimum limit point and all isolated points beneath this. 
First, we consider the case $\alpha_2>0$. 
%\begin{theorem}
%\label{thm:main0}
%We have $$
%\limsup_{n\to\infty}\|\xi\alpha^n\|=0
%$$
%if and only if $\xi\in X_1:=\frac {1}{\alpha-\alpha_2} \Z[\alpha]$. 
%\end{theorem}
\begin{theorem}
\label{thm:main}
Let $\alpha>3$ be a quadratic unit with $\alpha_2>0$. Define sequences of integers 
$(p_n)$  and $(q_n)$ by
$$
p_0=0,\ p_2=1, \ p_{2n+2}=b p_{2n}-p_{2n-2},\  p_{2n-1}=p_{2n-2}+p_{2n}
$$
and
$$
q_0=1,\ q_2=1+b, \ q_{2n+2}=b q_{2n}-q_{2n-2},\  q_{2n-1}=q_{2n-2}+q_{2n}
$$
for $n\in \N$.
%$$
%x_0=\frac {p_0}{q_0}= \frac 01, \quad
%x_{2n}=\frac{p_{2n}}{q_{2n}} \text{ for } n\in \N
%$$
%and
%$$
%x_{2n-1}=\frac{p_{2n}+p_{2n-2}}{q_{2n}+q_{2n-2}} \text{ for } n\in \N
%$$
%where $p_n/q_n$ is the $n$-th convergent of the continued fraction expansion 
%$$
%\frac {1}{1+\alpha}=[0;b,\overline{1,b-2}].
Then 
$(p_n/q_{n})_{n=0,1,\dots} $ is a strictly increasing sequence converging to $1/(1+\alpha)$ 
and we have
$$
\left\{\left. \limsup_{n\to \infty} \| \xi \alpha^n \| \ \right|\ \xi\in \R \right\} \cap 
\left[0,\frac 1{1+\alpha}\right] 
= \left\{\left. \frac {p_n}{q_{n}}\ \right| n=0,1,2,\dots \right\} \bigcup \left\{\frac 1{1+\alpha}\right\}.
%\left\{0<\frac {p_1}{p_2}<\frac {p_2}{p_3}<\dots <\frac{1}{1+\alpha}\right\}.
$$
Moreover, $X_n:=\{\xi\in \R\ \left|\ \limsup_{n\to \infty} \|\xi \alpha^n\|=p_n/q_{n} \right.\}$ 
is a subset of $\Q(\alpha)$, explicitly described in terms of 
Christoffel words, 
and
$$
X_{\infty}:=\left\{\xi\in \R\ \left|\ \limsup_{n\to \infty} \|\xi \alpha^n\|=\frac {1}{1+\alpha}\right.\right\}$$
is an uncountable set, described by eventually balanced words, a generalization of Sturmian words. 
%$$X_2=\frac {\alpha-1}{b^2-4} + X_1.$$
%For $\xi\in \R\setminus X_1$, we have
%$$
%\limsup_{n\to\infty}\|\xi\alpha^n\|\ge \frac 1{2+b}
%$$
%and the equality holds if and only if $\xi\in X_2$. 
%Further for $\xi\in \R\setminus (X_1\cup X_2)$, we have
%$$
%\limsup_{n\to\infty}\|\xi\alpha^n\|\ge \frac 1{1+\alpha}
%$$
%and the equality holds if and only if $\xi$ is contained in an uncountable set $X_3\subset \R$. 
%Finally we have
%\[
%\inf_{\xi\in \R\setminus (X_1\cup X_2\cup X_3)} 
%\limsup_{n\to\infty}\|\xi\alpha^n\|=\frac 1{1+\alpha}.
%\]
\end{theorem}

For the definition of Christoffel words and eventually balanced words, see \S \ref{LBW}. 
Consider the {\bf negative continued fraction} expansion 
$$\frac 1{1+\alpha}=\cfrac{1}{1+b-\cfrac {1}{b- \cfrac{1}{b- \cfrac {1}{\ddots} }}}
=: [0;1+b,b,b,b,\dots]_{neg}.$$
Using its $n$-th convergent $P_n/Q_n=[0,\underbrace{1+b,b,b,\dots,b}_{n}]_{neg}$,  we have 
$$\frac{p_{2n}}{q_{2n}}=\frac{P_{n}}{Q_{n}}\quad \text{ and }\quad \frac{p_{2n-1}}{q_{2n-1}}=\frac{P_{n-1}+P_{n}}{Q_{n-1}+Q_{n}}.$$

The complete description of $X_n$ will be given in the proof, 
but we introduce the first two of them:

\begin{cor}
\label{X2}
We have
$$
X_0=\frac {1}{\alpha-\alpha_2} \Z[\alpha],\qquad
X_1=\pm \frac {\alpha-1}{b^2-4} + X_0.$$ 
Therefore, we have
$$
\limsup_{n\to\infty}\|\xi\alpha^n\|\ge \frac 1{2+b}
$$
for $\xi\in \R\setminus X_0$, 
and 
$$
\limsup_{n\to\infty}\|\xi\alpha^n\|\ge \frac 1{1+b}
$$
for $\xi\in \R\setminus (X_0\cup X_1)$.
\end{cor}

Next, we consider the case $\alpha_2<0$. 

\begin{theorem}
Let $\alpha>3$ be a quadratic unit with $\alpha_2<0$.
Define sequences of integers 
$(p_n)$ and $(q_n)$ by
$$
p_0=0,\ p_2=b,\ p_{2n+2}=(b^2+2) p_{2n}-p_{2n-2},\ 
p_{2n-1}=p_{2n-2}+p_{2n}
$$
and
$$
q_0=1,\ q_2=b^2+3,\ q_{2n+2}=(b^2+2) q_{2n}-q_{2n-2},\ 
q_{2n-1}=q_{2n-2}+q_{2n}
$$
for $n\in \N$. Then 
$(p_n/q_{n})_{n=0,1,\dots}$ is a strictly increasing sequence converging to $b/(1+\alpha^2)$ 
and we have
$$
\left\{\left. \limsup_{n\to \infty} \| \xi \alpha^n \| \ \right|\ \xi\in \R \right\} \cap 
\left[0,\frac {b}{1+\alpha^2}\right] 
%= \left\{0<\frac {bp_1}{p_2}<\frac {bp_2}{p_3}<\dots <\frac{b}{1+\alpha^2}\right\}.
= \left\{\left. \frac {p_n}{q_{n}}\ \right| n=0,1,2,\dots \right\} \bigcup \left\{\frac {b}{1+\alpha^2}\right\}.
$$
Moreover, $X_n:=\{\xi\in \R\ \left|\ \limsup_{n\to \infty} \|\xi \alpha^n\|=p_n/q_{n}\right.\}$
is a subset of $\Q(\alpha)$, explicitly described in terms of 
Christoffel words, 
and
$$
X_{\infty}:=\left\{\xi\in \R\ \left|\ \limsup_{n\to \infty} \|\xi \alpha^n\|=\frac {b}{1+\alpha^2}\right.\right\}$$
is an uncountable set, described by eventually balanced words.

%Let \(\alpha\) be a quadratic unit with a negative conjugate. 
%Then 
%\[
%\inf_{\xi\in\mathbb{R}\atop \xi\not\in\mathbb{Q}(\alpha)}\limsup_{n\to\infty}\|\xi\alpha^n\|
%=
%\frac{\alpha^2-1}{\alpha(\alpha^2+1)}. 
%\]
\label{thm:main2}
\end{theorem}

Let $P_n/Q_n$ be the $n$-th convergent of the negative
continued fraction expansion 
$$
\frac 1{1+\alpha^2}
%=\frac{1}{b^2+3-\cfrac{1}{b^2+2-\cfrac{1}{b^2+2-\cfrac {1}{\ddots}}}}
=[0;b^2+3,b^2+2,b^2+2,b^2+2,\dots]_{neg}.
$$ 
Then we have $$p_{2n}=bP_{n},\quad q_{2n}=Q_{n},\quad 
p_{2n-1}=p_{2n-2}+p_{2n}, \text{ and } q_{2n-1}=q_{2n-2}+q_{2n}.$$
Note that $p_n$ and $q_{n}$ are not always coprime.

\begin{cor}
We have
$$X_0=\frac {1}{\alpha-\alpha_2} \Z[\alpha],\qquad
X_1=\pm \frac {\alpha}{4+b^2} + X_0.$$
Therefore we have
$$ 
\limsup_{n\to\infty}\|\xi\alpha^n\|\ge \frac {b}{4+b^2}
$$
for $\xi\in \R\setminus X_0$,
and 
\[
\limsup_{n\to\infty}\|\xi\alpha^n\|\ge \frac {b}{3+b^2}
\]
for $\xi\in \R\setminus (X_0\cup X_1)$.
\end{cor}

In Theorems \ref{thm:main} and \ref{thm:main2}, we have  $X_n\subset \Q(\alpha)$ for $n=0,1,\dots$ and 
$X_{\infty}$ contains a transcendental number. Thus, we get a %we have
\begin{cor}
Let $\alpha>3$ be a quadratic unit. Then
$$
\inf_{\xi\in \R\setminus \Q(\alpha)} 
\limsup_{n\to\infty}\|\xi\alpha^n\|
=\begin{cases} 1/(1+\alpha)%\frac 1{1+\alpha} 
& \text{ for } \alpha_2>0\\
b/(1+\alpha^2)%\frac{b}{1+\alpha^2}
& \text{ for } \alpha_2<0.\end{cases}
$$
\end{cor}

\section{Characterization of balanced words}
\label{LBW}
We introduce the second key of this paper: % in \S \ref{LBW}
a characterization of the set of bi-infinite balanced words.
%We give a description of it as a subshift, i.e., bi-infinite words described by 
%a special set of forbidden words in \S \ref{LBW}. 
%We believe that this characterization is interesting in itself.
Our proof of Theorems \ref{thm:main} and \ref{thm:main2} will be finished by using an unavoidable 
symmetry of {\bf balanced words} or {\bf sturmian words} in \S \ref{Proofs} .

Let $\A=\{0,1\}$ and
% and denote by $\A^*$ the monoid generated by $\A$ by 
%concatenation as a binary operation. 
%The empty word $\lambda$ is the identity and put $\A^+=\A^*\setminus \{\lambda\}$.
\(v=v_1v_2\ldots v_n\) be a non empty word on \(\A\). Write the length of 
\(v\) by \(|v|=n\). The mirror word of $v$ is denoted by 
\(\widetilde{v}=v_n v_{n-1}\ldots v_1\) and set \(\widetilde{\lambda}=\lambda\). 
Put \(v^0=\lambda\) and $v^k=\underbrace{vv\ldots v}_k$ for $k\in \N$. 
%Moreover, for any real number \(x\), put 
%\begin{eqnarray*}
%v^x = v^{\lfloor x \rfloor} v',
%\end{eqnarray*}
%where \(v'\) is a prefix of \(v\) with length \(\lceil\{x\} v\rceil\). 
We are interested in a subshift of the full-shift $\A^{\Z}$ (resp $\A^{\N}$). 
Set $F=\{ 0v01\tilde{v}1, 1\tilde{v}10v0\ |\ v\in \A^* \}$ and 
we study the subshift $\X_F$ (resp $\Y_F$) of $\A^{\Z}$ (resp $\A^{\N}$) defined by 
the set $F$ of forbidden words. 
In this section, we prove Theorem \ref{Balanced}, a restatement of Reutenauer \cite[Theorem 3.1]{Reutenauer:06},
which asserts that $\X_F$ is the set of bi-infinite balanced words.
The proof is different from \cite{Reutenauer:06, Yasutomi} and self-contained.
This formulation admits a one-sided version (Theorem \ref{Balanced2}) as well, which will be used 
in the proof of Theorem \ref{thm:main}.

Clearly $\X_F$ and $\Y_F$ are invariant by the involution $0 \mapsto 1, 1\mapsto 0$. 
Moreover, $\X_F$ is invariant by the mirror map $\x\mapsto \tilde{\x}$. 
We use the standard notation in language theory, e.g., 
$(1+10)^+$ is the set of non empty finite words generated by $1$ and $10$. 
$B^{\N}$ (resp. $B^{\Z}$) is the set of
right infinite words (resp. bi-infinite words) generated by
$B\subset \A^*$. %, and $B^{\Z}$ is the bi-infinite words by $B$.
For $v\in \A^+$, we denote by $v^{\N}$ (resp. $v^{\Z}$) the right infinite (resp. bi-infinite) word 
concatenating $v$. 
For $u, v\in \A^*$ we denote by $u\prec v$ if $u$ is a factor (subword) of
$v$.  This notation naturally extends to $\x\in \A^{\Z}$, i.e., $u\prec \x$ if
$u\in \A^*$ is a factor of $\x$.  
As usual, $u\not \prec v$ means $u\prec v$ does not hold.
For $\x=(x_i)_{i\in \Z}\in \A^{\Z}$, we
write $\x[i,j]=x_i\dots x_j$ for $i\le j$ and $\x[i,j]=\lambda$ if $i>j$. 
The relation $\prec$ does not contain 
information on the position of the factor except when we write 
$w_1\x[i,j]w_2\prec \x$ with $w_1, w_2\in \A^*$; it means $w_1=\x[\ell,i-1]$ and $w_2=\x[j+1,k]$
with some $\ell$ and $k$. 

\begin{lem}
\label{Diff}
If $10^a10^b1$ is a factor of a word $\x$ in $\X_F$, %If $10^a10^b1\prec \x\in \X_F$, 
then $|a-b|\le 1$.
\end{lem}

\begin{proof}
If $b\ge a+2$, then $10^{a}{\bf 10}0^{a+1}\in F$ is a factor of $\x$. The case $b\le a-2$ is similar.
\end{proof}
\begin{lem}\label{main1}
Take $\x\in \X_F$.
Assume that $\x[i,j]=(10^a)^n10^{a+1}10^a1(0^a1)^n\prec \x$ 
or $\x[i,j]=(10^{a+1})^n10^{a}10^{a+1}1(0^{a+1}1)^n\prec \x$ with $n\in \N$. 
Then $0^a\x[i,j]0^a\prec \x$ , $0^{a+2}\x[i,j] \not \prec \x$ and
 $\x[i,j]0^{a+2} \not \prec \x$. 
\end{lem}

\begin{proof}
Let $n=0$ and $\x[i,j]=10^{a+1}10^a1\prec \x$. 
By Lemma \ref{Diff}, $0^a$ must be prepended to $\x[i,j]$. 
Thus it must be followed by $0^a$ by considering 
 the forbidden words centered by ${\bf 01}$ around $0^a10^{a}{\bf 01}0^a1$, i.e., 
 $0^a\x[i,j]0^a\prec \x$\textcolor{red}{. 
By Lemma} \ref{Diff},  $\x[i,j]0^{a+2} \not \prec \x$.  If
$0^{a+2}\x[i,j] \prec \x$ then $\x[i,j]$ must be followed by $0^{a+2}$ \textcolor{red}{by considering 
 the forbidden words centered by ${\bf 01}$ around $0^{a+2}10^{a}{\bf 01}0^a1$, i.e., 
 $0^{a+2}\x[i,j]0^{a+2}\prec \x$}, 
which is absurd.
Now let $\x[i,j]=(10^a)^n10^{a+1}10^a1(0^a1)^n\prec \x$ with $n\ge 1$.
We have $$\x[i,j]=(10^a)^n{\bf 10} (0^a1)^{n+2}=(10^a)^{n+1}{\bf 01}(0^a1)^{n+1}.$$ 
Considering
the forbidden words centered at ${\bf 10}$, we see $0^a\x[i,j]\prec \x$. Switching focus to 
the forbidden words centered at ${\bf 01}$, we see $0^a\x[i,j]0^a\prec \x$.
Clearly $0^{a+2}\x[i,j] \not \prec \x$ and
 $\x[i,j]0^{a+2} \not \prec \x$ by Lemma \ref{Diff}.

Let $n=0$ and $\x[i,j]=10^{a}10^{a+1}1\prec \x$. 
Then $\x[i,j]0^a\prec \x$ by Lemma \ref{Diff}.  Considering 
$10^{a}{\bf 10} 0^a1$ centered at ${\bf 10}$, we have $0^a\x[i,j]0^a\prec \x$.
Since $0^{a+2}\x[i,j] \not \prec \x$, $\x[i,j]0^{a+2} \prec \x$ leads to 
$0^{a+2}\x[i,j]0^{a+2} \prec \x$, a contradiction.
Now we assume 
$\x[i,j]=(10^{a+1})^n10^{a}10^{a+1}1(0^{a+1}1)^n\prec \x$ 
with $n\ge 1$. By Lemma \ref{Diff}, we have $0^a\x[i,j]0^a\prec \x$. We use two centers
of the forbidden words: 
$$\x[i,j]=(10^{a+1})^{n-1}10^{a}{\bf 01}0^a1(0^{a+1}1)^{n+1}
=(10^{a+1})^{n}10^{a}{\bf 10}0^a1(0^{a+1}1)^{n}.$$  
By focusing on ${\bf 01}$, we see $0^{a+2}\x[i,j]\not \prec \x$. 
Thus by focusing on $\bf{10}$, 
$\x[i,j]0^{a+2}\prec \x$  implies $0^{a+2}\x[i,j]0^{a+2}\prec \x$ which is impossible.
\end{proof}

Lemma \ref{main1}
shows that the two factors in question 
must be continued to the 
right only by $0^a1$ or $0^{a+1}1$, and to the left only by $10^a$ or $10^{a+1}$.
We have further restrictions.

\begin{lem}\label{main2}
Assume that $\x\in \X_F$.
If $\x[i,j]=(10^a)^n10^{a+1}10^a1(0^a1)^{n+1}\prec \x$, then 
we have $$10^a\x[i,j]=
(10^a)^{n+1}10^{a+1}10^a1(0^a1)^{n+1}\prec \x.$$ 
If $\x[i,j]=(10^{a+1})^n10^{a}10^{a+1}1(0^{a+1}1)^{n+1}\prec \x$, 
then $$
10^{a+1}\x[i,j]=(10^{a+1})^{n+1}10^{a}10^{a+1}1(0^{a+1}1)^{n+1}\prec \x.$$
\end{lem}

\begin{proof}
Assume that
$\x[i,j]=(10^a)^{n+1}{\bf 01}(0^a1)^{n+2}\prec \x$. By Lemma \ref{main1},
$0^a\x[i,j]\prec \x$ and considering the forbidden words centered at ${\bf 01}$ we see
$10^a\x[i,j]\prec \x$. 
If $\x[i,j]=(10^{a+1})^{n}10^{a}{\bf 10}0^a1(0^{a+1}1)^{n+1}$, then Lemma \ref{main1} implies
$0^{a}\x[i,j]\prec \x$ and the forbidden words centered at ${\bf 10}$ shows 
$10^a\x[i,j]\not \prec \x$.  Therefore $10^{a+1}\x[i,j]\prec \x$.
\end{proof}

\begin{lem}\label{main3}
If  $10^a10^{a+1}1$ or $10^{a+1}10^a1$ is a factor of $\x\in \X_F$, 
then $\x\in (1(0^a+0^{a+1}))^{\Z}$.
\end{lem}

\begin{proof}
Assume that $10^{a+1}10^{a}1$ is a factor. 
By Lemmas \ref{main1} and \ref{main2},  unless  the right prolongation is 
$10^{a+1}1(0^{a}1)^{\infty}$, we find a word in $10^{a+1}1(0^a1)^+0^{a+1}1$ whose 
suffix is $10^{a}10^{a+1}1$ in the right extension. 
The same goes for $10^{a}10^{a+1}1$ that unless $10^{a}1(0^{a+1}1)^{\infty}$ is the right 
continuation,
we find a word in $10^a1(0^{a+1}1)^+0^a1$ having a suffix $10^{a+1}10^{a}1$ in the right. 
By iterating this discussion, the right extension must be 
in $(1(0^a+0^{a+1}))^{\N}$. The left direction is similar.
\end{proof}

Let $K$ be the set of elements $\x\in \A^{\Z}$ that have at most one occurrence of $0$'s or $1$'s, i.e., 
the set of words of the form $0^{\Z}$, $1^{\Z}$, $\dots 000 1 000 \dots$ or $\dots 111 0 111 \dots$. 
Clearly $K\subset \X_F$.  On the other hand,  it is easy to see that 
if $\x\in \A^{\Z}$ has exactly two occurrences of $0$'s or $1$'s, then $\x\not \in \X_F$. 
By Lemmas \ref{Diff} and \ref{main3}, any element $\x\in \X_F\setminus K$ must be in %of the form
$(1(0^a+0^{a+1}))^{\Z}$ or $(0(1^a+1^{a+1}))^{\Z}$ with $a\ge 1$. 
Therefore for $x\in \X_F\setminus K$, we associate a bi-infinite sequence over two letters %alphabets
$a$ and $a+1$ by choosing an appropriate index\footnote{We select the origin of the word $\{a,a+1\}^\Z$
by the position where the origin of $\x$ sits in $10^a$ or $10^{a+1}$, say.}. 
Replacing $a$ to $0$ and $a+1$ to $1$, we define a map 
$\phi$ from $\X_F\setminus K$ to $\A^{\Z}$.\footnote{$(01)^{\infty}$ is the unique word in $\X_F\setminus K$ that is in $(1(0^1+0^{2}))^{\Z}$ and $(0(1^1+1^{2}))^{\Z}$. We put $\phi((01)^{\infty}):=0^{\infty}$.} 

\begin{prop}\label{InverseLimit}
The image of $\phi$ is contained in $\X_F$. %$\phi$ is a map from $\X_F\setminus K$ to $\X_F$.
\end{prop}

\begin{proof}
Assume that a  word $0v01\tilde{v}1\in F$ appeared as 
a factor of $\phi(\x)$ with $\x\in \X_F\setminus K$. Then there is a factor
$$
10^a10^{a_1}\dots 10^{a_{\ell}} 10^a {\bf 10} 0^{a}1 0^{a_{\ell}}1 \dots 0^{a_1}1 0^{a+1}
$$
of $\x$ with $a_i\in \{a, a+1\}$. However it is a forbidden word, giving a contradiction.
The case of $1\tilde{v}10v0$ is similar.
\end{proof}

An element $\x=(x_i)_{i\in \Z} \in \A^{\Z}$ is positively (resp. negatively) eventually periodic 
if $(x_i)_{i\in \N}$ (resp. $(x_{-i}))_{i\in \N}$) is eventually periodic.
We say that $\x\in \A^{\Z}$ is eventually periodic if it is positively or 
negatively eventually periodic. 
However this distinction does not exist for the elements in $\X_F$

\begin{lem}\label{Periods}
If $\x\in \X_F$ is positively eventually periodic then it is negatively eventually periodic and vice versa.
\end{lem}

\begin{proof}
If $\x=(x_i)\in K$, the statement is trivial.
If not, by switching $0$ and $1$ if necessary, 
we may assume that $\x\in (1(0^a+0^{a+1}))^{\Z}$ with $a\ge 1$
and the period in $(x_i)_{i\in \N}$  is of the form $10^{a_1}10^{a_2}\dots 10^{a_{\ell}}$ with $a_i\in \{a, a+1\}$. 
Applying $\phi$ the period length $\ell$ becomes strictly shorter.  
Therefore eventually we find $n\in \N$ that $\phi^n(\x)$ has a period of length one in positive direction. 
In view of Lemma \ref{main2}, we see $\phi^n(\x)\in K$ which is negatively eventually periodic.
Considering the inverse image $\phi^{-n}(\y)$ for $\y\in K$, we obtain the result.
The converse direction is shown in a similar manner.
\end{proof}

Later we will see that Lemma \ref{Periods} corresponds to Lemma 6.2.2 in \cite{Fogg:02}.
The same proof shows that for every
eventually periodic element $\x\in \X_F$, there is $n\in \N$ that $\phi^n(\x)\in K$.

A word $\x\in \A^* \cup \A^{\N} \cup \A^{\Z}$ is {\bf balanced} if $||u|_1-|v|_1|\le 1$ holds 
for all $u,v\prec \x$ with $|u|=|v|$, where $|u|_1$ denotes the number of occurrences of $1$ in $u$.
An infinite word $\y:=(y_m)\in \A^{\N}$ is {\bf eventually balanced}, if there exist
$m_0\in \N$ and a function $q:\N \to \N$ with $\lim_{m\to \infty} q(m)=\infty$ 
so that
$$ \y[\max\{1,m-q(m)\},m+q(m)]
$$
is balanced for all $m\ge m_0$. 

\begin{lem}
\label{Unbalance}
If a word $\x\in \A^* \cup \A^{\N} \cup \A^{\Z}$ is not balanced then there exist a palindrome $w\in \A^*$ 
such that $0w0$ and $1w1$ are the factors of $\x$.
\end{lem}

\begin{proof}
Take $u,v \prec \x$ with $|u|=|v|$,  $|v|_1-|u|_1\ge 2$ and $|u|$ being minimum. 
Then $u,v$ must have the desired form (see \cite[Proposition 2.1.3]{Lothaire:02}).
\end{proof}

\begin{theorem}\label{Balanced}
$\X_F$ is the set of balanced words in $\A^{\Z}$.
\end{theorem}

\begin{proof}
Since every element of $F$ is not balanced,  all balanced words are in $\X_F$.  Clearly the element in $K$ is
balanced. Let $\x\in \X_F\setminus K$ and assume that there is a palindrome $w$ that
 $0w0$ and $1w1$ are the factors of $\x$.  If $|w|=0$, then $00$ and $11$ are the factors of $\x$, which already 
 contradicts Lemma \ref{main3}, i.e., $\x$ can not be in $(1(0^a+0^{a+1}))^{\Z}$ or
 $(0(1^a+1^{a+1}))^{\Z}$.   If $w\in 0^+$ or $w\in 1^+$, then we get the same contradiction.
Since $w$ is a palindrome, $|w|\le 2$ is impossible.
Assume that $|w|\ge 3$ and $w\not \in 0^+$ and $w\not \in 1^+$. Without loss of generality, we may assume
that $w$ is of the shape $0^{a}1 \dots 10^{a}$. Thus, since 
$10^{a}1 \dots 10^{a}1$ and $0^{a+1}1\dots 10^{a+1}$ are factors of $\x$, %in $\phi(\x)$
there is a palindrome
$w'$ such that $0w'0$ and $1w'1$ are factors of $\phi(\x)$ and $|w'|<|w|$.  Iterating this we reach a
contradiction in finite steps.
\end{proof}

%\begin{rem}
%\begin{rm}
%\label{Coin}
%To analyse classical Lagrange spectrum, 
%Bombieri \cite[Lemma 11,13]{bom} defined a set $B$ of infinite words over two letters. 
%We can directly show that $B=\X_F$ by identifying these letters with $0,1$. 
%Moreover, we see that 
%$B$ is invariant under Sturmian morphism (Nielsen automorphism) and our map $\phi$ 
%is a certain iterate of inverse sturmian morphism.
%\end{rm}
%\end{rem}

For $\Y_F$, the story goes in a similar but a little more intricate way. 
By abuse of notation, let $K$ be the set of elements $\y\in \A^{\N}$ 
that have at most one occurrence of $0$'s or $1$'s, i.e., 
the set of words of the form $0^{\N}$, $1^{\N}$, $0^* 1 0^{\N}$ or $1^*01^{\N}$. 
Clearly $K\subset \Y_F$. Suppose $\y\in \A^{\N}$ has a suffix $0^{\N}$. If $\y$ has more than 
one occurrences of $1$, then $\y$ has a forbidden factor $10^n10^{n+2}$ and it is not in $\Y_F$.
Thus if $\y\in \Y_F$ has  a suffix $0^{\N}$ or $1^{\N}$, then it must be in $K$.
For a $\y\in \Y_F\setminus K$, considering the involution $0\mapsto 1,1\mapsto0$, we 
may assume that $\y=0^{\ell}10^{a_1}10^{a_2}\dots$ and $a_i\ge 1, a_{i+1}\ge 1$ for infinitely many $i$'s.
Assume that there are infinitely many indices $i$ that $a_i\neq a_{i+1}$.
Following the proof of Lemma \ref{main3}, 
we find an infinite sequence of prefixes of $\y$ of the form
$w_i\in 0^* (1(0^{b_i}+0^{b_i+1}))^{+}$ with $b_i=\min\{a_i,a_{i+1}\}$
where $|w_i|$ is the index of the last $0$ of $0^{a_{i+1}}$.
Clearly $b_i$ does not depend on the index $i$
and $\y \in 0^*(1(0^a+0^{a+1}))^{\N}$. 
Observing the exponents of $0$ in $\y$, we obtain a sequence over $a$ and $a+1$.
Substituting $a$ by $0$, $a+1$ by $1$, we see the map $\phi:\Y_F \setminus K \rightarrow \A^{\N}$ 
is defined in a similar manner.
Consider the case that there is $m\in \N$ satisfying $a_i=a_{i+1}$ for $i\ge m$. If there exists an index $i<m$
that $a_i\neq a_{i+1}$, then take the largest $i=i_0$ with this property. 
Let $a=\min\{a_{i_0},a_{i_0+1}\}$. By Lemmas \ref{main1} and \ref{main2}, we have
$\y\in 0^{*}(10^a)^+10^{a+1}(10^a)^{\N}$ or
$\y\in 0^{*}(10^{a+1})^+10^{a}(10^{a+1})^{\N}$. Thus in all cases, $\y \in 0^*(1(0^a+0^{a+1}))^{\N}$
and we can define $\phi(\y)$. After all we obtain in the same way:

\begin{prop}\label{InverseLimit2}
$\phi$ is a map from $\Y_F\setminus K$ to $\Y_F$.
\end{prop}

And once this map is defined, we obtain

\begin{theorem}\label{Balanced2}
$\Y_F$ is the set of balanced words in $\A^{\N}$
\end{theorem}

\noindent
by the same proof. It is easy to see from the proof
that the projection
$(x_i)_{i\in \Z}\mapsto (x_i)_{i\in \N}$ gives a surjective map from $\X_F$ to $\Y_F$.

\begin{cor}\label{Unavoidable}
If $\y\in \A^{\Z} \cup \A^{\N}$ is not balanced, then there is a word $v\in \A^*$ that
$0v01\tilde{v}1$ or $1v10\tilde{v}0$ is a factor of $\y$.
\end{cor}

Note that the same statement no longer holds for finite words. 
Indeed $w=1010010001$ is not balanced but
there is no factor of $w$ in $F$. This is a characteristic difference from Lemma \ref{Unbalance}. 
\bigskip

Let $p(\x,n)$ be the number of distinct factors of $\x\in \A^{\N} \cup \A^{\Z}$ of length $n$. The word 
$\x\in \A^{\N}$
is {\bf sturmian} if $p(\x,n)=n+1$ for all $n=1,2,\dots$. Theorem of Morse-Hedlund asserts that
$\x\in \A^{\N}$ is sturmian if and only if $\x$ is not eventually periodic and balanced 
(See \cite[Theorem 2.1.5]{Lothaire:02}
and \cite[Theorem 6.1.8]{Fogg:02}). 
For $\x\in \A^{\Z}$, there is an eventually period word 
$\y=\dots 000111\dots$ which satisfies $p(\y,n)=n+1$ but not balanced. 
For a bi-infinite word $\x\in \A^{\Z}$, we define $\x$ is {\bf sturmian}
if it is not eventually periodic and $p(\x,n)=n+1$ for all $n\in \Z$. 
Under this definition, a sturmian word is a non eventually periodic balanced word and vice versa
 (see \cite[Proposition 6.2.5]{Fogg:02}).

%\begin{cor}\label{Sturmian}
%A non eventually periodic word $\x\in \X_F \cup \Y_F$ is sturmian.
%\end{cor}

If $\x\in \X_F \cup \Y_F$ is not eventually periodic, then $\phi(\x)\in  \X_F\cup \Y_F$ is so. 
One can iterate $\phi$ infinitely many times to $\x$. On the other hand, if $\x$ is eventually
periodic, then there is $n\in \N$ that $\phi^n(\x)\in K$, because the period length decreases by
$\phi$. 

For $\alpha\in [0,1]$ and $\beta\in \R$, a lower mechanical word in $\A^{\N} \cup \A^{\Z}$ 
is 
$$ (\lfloor \alpha(n+1)+\beta\rfloor -\lfloor \alpha n+\beta\rfloor) $$
with $n\in \N$ or $n\in \Z$. Its slope is $\alpha$ and intercept $\beta$. 
Upper mechanical word is defined by replacing $\lfloor \cdot \rfloor$ by $\lceil \cdot \rceil$.
A mechanical word is either a lower or an upper mechanical word.
A sturmian word $\x\in \A^{\N} \cup \A^{\Z}$ is characterized as a mechanical word with an
irrational slope $\alpha$ (see \cite[Theorem 2.1.13]{Lothaire:02} and \cite[Chapter 6]{Fogg:02}).
Its slope $\alpha$ is computed as a frequency of $1$, i.e., 
$$\alpha=\alpha(\x)=\lim_{n\rightarrow \infty} \frac{|\x[0,n-1]|_1}{n}.$$
For a sturmian word $\x$, we may assume that the frequency of $0$ is larger than that of $1$.
Our map $\phi$ is conjugate to a continued fraction algorithm acting on slopes, that is, 
$$
\alpha(\phi(\x))= \frac 1{\alpha(\x)} - \left\lfloor \frac 1{\alpha(\x)} \right\rfloor.
$$
Let $
\alpha(\x)=[0;a_1,a_2,a_3,\dots]
%=\cfrac{1}{a_1+\cfrac{1}{a_2+\cfrac{1}{a_3+\cfrac{1}{\ddots}}}}
$
be the continued fraction expansion of $\alpha$.
Then the sturmian sequence $\phi^{n-1}(\x)$ lies in $0^*(1(0^{a_n}+0^{a_n+1}))^{\N}$.
In other words, the map $\phi$ is the shift acting on the first coordinate of the
multiplicative coding of sturmian words (see \cite[Chapter 6]{Fogg:02}).

Let $\iota(\x)=\sup\{ n\in \N\ |\ |v|=n,\ v01\tilde{v}\prec \x \text{ or } v10\tilde{v}
\prec \x \}$. Every mechanical word 
$\x\in \A^{\N}$ is equivalent to a cutting sequence of integer coordinate grids in $\R^2$ 
by a line (see \cite[Remark 2.1.12]{Lothaire:02}). 
As the line must pass an arbitrary small neighborhood of some lattice point, 
we immediately see arbitrary large symmetric coding 
centered at $01$ or $10$ around the lattice point, i.e., $\iota(\x)=\infty$.
Therefore every non eventually periodic 
sturmian word $\x\in \X_F$, we have $\iota(\x)=\infty$
as well. Here we give a little more general statement with a simple proof.

\begin{prop}\label{Iota}
For every $\x\in \X_F \cup \Y_F$, we have $\iota(\x)<\infty$ if and only if
$\x$ is a purely periodic balanced word.
\end{prop}

In other words, $\iota(\x)=\infty$ if and only if $\x$ is a sturmian word or
a balanced periodic word which is not purely periodic.

\begin{proof}
For a non eventually periodic $\x$, we can apply $\phi$ infinitely many times. 
If $\iota(\x)\le 1$, then $10^a10^{a+1}1$ or $10^{a+1}10^a1$ with $a\ge 1$
can not be a factor of $\x$. 
The same is true for  $01^a01^{a+1}0$ or $01^{a+1}01^a0$ with $a\ge 1$ and
$\x$ must be eventually periodic.
Assume that $2\le \iota(\y)<\infty$. We easily see $\iota(\phi(\y))<\iota(\y)$. For example, 
if the central part of 
$$10^{a}10^{a_1}10^{a_{2}}1\dots 10^{a_{\ell}} 10^{a} {\bf 01} 0^{a}1 0^{a_{\ell}}
1 \dots 10^{a_2} 10^{a_1}10^{a+1}$$
attains $\iota(\y)$ 
with $a_i\ge 1, a\ge 1$, then
$$
\phi(\y)=\tau(a_1\dots a_{\ell}(a+1)a a_{\ell}\dots a_1)
$$
where $\tau: \{a,a+1\}^* \rightarrow \{0,1\}^*$ is the morphism defined by $\tau(a)=0, \tau(a+1)=1$. 
Thus $\iota(\y)=2a+1+\sum_{i=1}^{\ell} (a_i+1)>\iota(\phi(\y))=\ell$.
Iterating $\phi$ we get a contradiction.

If $\x$ is eventually periodic, then there exists $n\in \N$ that $\phi^n(\x)\in K$.
We claim that $\iota(\x)<\infty$ if $\x$ is purely periodic. 
In fact, assume that $\iota(\x)=\infty$ and let $p$ be the period of $\x$. 
Take $v$ with $|v|>p$
that $v01\tilde{v}$ (or $v10\tilde{v}$) is a factor of $\x$ 
then we find that $v=w01u$ with $w,u\in \A^*$
that $u$ is a palindrome and we have $\x=((01)u)^{\Z}$. However we have $\iota(\x)=|u|<\infty$, 
a contradiction. This shows the claim.
In the case that $\x\in \phi^{-n}(\y)$ with $\y\in \{0^{\Z},1^{\Z}\}$, $\x$
is purely periodic and $\iota(\x)<\infty$. 
Otherwise $\y\in \{0^{\infty}10^{\infty}, 1^{\infty}01^{\infty}\}$
and we see $\iota(\x)=\infty$, since $\iota(\phi(\x))<\iota(\x)$ and $\iota(\y)=\infty$.
\end{proof}

\begin{ex}\label{TM}
Let $\x=0110100110010110\dots \in \A^{\N}$ 
be the fixed point of the Thue-Morse substitution $0\rightarrow 01, 1\rightarrow 10$.
Since $0011\prec \x$, it is not balanced. We can confirm that $\iota(\x)=6$ and
$F\ni 0100110{\bf 01}0110011, 0011001{\bf 01}1001101 \prec \x$.
\end{ex}

A {\bf Christoffel word} $\nu\in \A^*$ is
the period of a mechanical word of rational slope and intercept $0$.
For coprime integers $p,q$
with $0\le p\le q$, we define the lower Christoffel word of slope $p/q$ by
$$
\left(\left\lfloor \frac {p(i+1)}q \right\rfloor-\left\lfloor \frac {pi}q \right\rfloor\right)_{i=0,1,\dots,q-1}\in \A^q
$$
and replace $\lfloor\cdot\rfloor$ by $\lceil\cdot\rceil$ to get the upper Christoffel word.
%$$
%\left(\left\lceil \frac {p(i+1)}q \right\rceil-\left\lceil \frac {pi}q \right\rceil\right)_{i=0,1,\dots,q-1}\in \A^q.
%$$
A purely periodic balanced word is written as $\nu^{\Z}$
with a Christoffel word $\nu$ (see \cite[Theorem 2.1,3.2,4.1]{Berstel-ADLuca:97}).
A comprehensive survey on Christoffel words is found in \cite{Bersteletal:09}. 
Standard words can be defined by
 a slight modification of Christoffel words, switching indices from $i=0,\dots,q-1$ to $i=1,\dots,q$. 
An equivalent definition of standard words by a generating binary tree is found in \cite{Lothaire:02},
which is related to trees emerged in the study of the classical Markoff-Lagrange spectrum.
We shall use two basic properties of Christoffel words:

\begin{lem}
\label{Central}
A Christoffel word of length greater than $1$ is of the form $0v1$ or $1v0$ with a palindrome $v$.
\end{lem}

\begin{proof}
This directly follows from the definition. 
See \cite[Formula (2.1.14)]{Lothaire:02} and \cite[Proposition 4.2]{Bersteletal:09}.
\end{proof}

The palindrome $v$ appeared in the Christoffel word $\nu$ is called a {\bf central word}, which 
plays an important role (see \cite{Lothaire:02}).

%We say that a word $w\in \A^+$ appears at the origin of a bi-infinite word $\x$, if
%$\x[1,|w|]=w$. 

\begin{lem}
\label{Maxiota}
Let $\nu$ be a Christoffel word and put $\x=\nu^{\Z}$.
Then $\iota(\x)=|\nu|-2$. 
\end{lem}

\begin{proof}
Lemma \ref{Central} shows $\iota(\x)\ge |\nu|-2$. 
Let $\y=\sigma^n(\x)$ and $\y[0,1]=ab$ with $\{a,b\}=\{0,1\}$. 
From periodicity, we see $\y[1-|\nu|,1-|\nu|]=b$ and $\y[|\nu|,|\nu|]=a$. This implies $\iota(\x)\le |\nu|-2$.
\end{proof}

\begin{rem}\label{rem:ch}
In Lemma \ref{Maxiota}, $\iota(\x)$ is the maximum length of a factor $v$ that $\tilde{v}abv\prec \x$
with $\{a,b\}=\{0,1\}$. The equality $|v|=\iota(\x)(=|\nu|-2)$ 
is attained at exactly two indices $n$ modulo $|\nu|$. 
They correspond to the upper and lower Christoffel words (see Appendix).
\end{rem}

\section{Proof of Theorems \ref{thm:main} and \ref{thm:main2}}
\label{Proofs}
\subsection{Preparation}
If \(\alpha_2=\alpha^{-1}\), then 
\begin{eqnarray}\label{Eqn:10}
\frac{1}{\alpha+1}=(\ldots001.\overline{1}00\ldots)_{\alpha}
=(0^{\infty}1. \overline{1} 0^{\infty})_{\alpha},
\end{eqnarray}
where \(b^{\infty}\) denotes the infinite sequence of \(b\). %and \(\overline{1}\) means \(-1\). 
For simplicity, we write 
\begin{eqnarray*}
(0^{\infty}y_n y_{n-1}\ldots y_0. y_{-1} y_{-2}\ldots y_{-m} 0^{\infty})_{\alpha}
=:(y_n y_{n-1}\ldots y_0. y_{-1}y_{-2}\ldots y_{-m})_{\alpha}.
\end{eqnarray*}
In particular, (\ref{Eqn:10}) implies that 
\[
\frac{1}{\alpha+1}=(01.\overline{1})_{\alpha}=(\overline{1}1.0)_{\alpha}.
\]
Moreover, if \(\alpha_2<0\), then 
\begin{align}\label{eqn:basic2}
\frac{b}{\alpha^2+1}=(1.0\overline{1})_{\alpha}=(\overline{1}01.0)_{\alpha}.
\end{align}
Let \(\mathbf{y}=(y_n)_{n=-\infty}^{\infty}\) and 
\(\mathbf{y}'=(y_n')_{n=-\infty}^{\infty}\) be sequences of integers such that 
\(|y_n|,|y_n'|\leq 1\) for each \(n\). 
Now we define the sequence \(\underline{\psi}(\mathbf{y})
=(\psi_n(\mathbf{y}))_{n=0}^{\infty}\) as follows: 
If $\alpha_2=\alpha^{-1}$, then we put 
\begin{eqnarray*}
\psi_0(\mathbf{y}):=y_0
\mbox{, }
\psi_n(\mathbf{y}):=y_n+y_{-n}
\mbox{ for }
n=1,2,\ldots.
\end{eqnarray*}
If $\alpha_2=-\alpha^{-1}$, then let 
\begin{eqnarray*}
\psi_0(\mathbf{y}):=y_0
\mbox{, }
\psi_n(\mathbf{y}):=(-1)^n y_n+y_{-n}
\mbox{ for }
n=1,2,\ldots.
\end{eqnarray*}
We denote 
\(\mathbf{y}\gg \mathbf{y}'\) if 
%\(\underline{\psi}(\mathbf{y})\) is greater than 
%\(\underline{\psi}(\mathbf{y}')\) with respect to the lexicographical order. 
%Namely, 
there is a nonnegative integer \(l\) such that, for \(h=0,1,\ldots, l-1\), 
\(\psi_h(\mathbf{y})=\psi_h(\mathbf{y}')\) and that 
\(\psi_l(\mathbf{y})>\psi_l(\mathbf{y}')\) and $|\psi_n(\mathbf{y})-\psi_n(\mathbf{y}')|\le 2$ for
$n\ge l$. 
\begin{lem}
If \(\mathbf{y}\gg\mathbf{y}'\), then 
\((\mathbf{y})_{\alpha}>(\mathbf{y}')_{\alpha}\).
%If 
%\(\mathbf{y}\gg\mathbf{y}'\) or 
%\(\underline{\psi}(\mathbf{y})=
%\underline{\psi}(\mathbf{y}')\), then 
%\(\mathbf{y}_{\alpha}\geq\mathbf{y}_{\alpha}'\).
\label{lex}
\end{lem}
\begin{proof}
Suppose that \(\mathbf{y}\gg \mathbf{y}'\). Let 
\(%\begin{eqnarray*}
l=\min\{h\geq 0 \mid y_h \ne y_h'\}.
\) %\end{eqnarray*}
Then, since \(\alpha> 3\) and \(|\psi_n(\mathbf{y})-\psi_n(\mathbf{y}')|\leq 2\) for each \(n\), 
we get 
\begin{align*}
(\mathbf{y})_{\alpha}-(\mathbf{y}')_{\alpha}
&\geq
\frac{1}{\alpha-\alpha_2}\left(
\alpha^{-l}-\sum_{h\geq l+1}\alpha^{-h}\cdot 2
\right)\\
&\geq \frac{\alpha^{-l}}{\alpha-\alpha_2}\left(
1-2\sum_{h\geq 1}\alpha^{-h}
\right)
>0.
\end{align*}
%Similarly, if \(\mathbf{y}\succ\mathbf{y}'\), then 
%\(\mathbf{y}_{\alpha}\geq\mathbf{y}_{\alpha}'\).
\end{proof}

\subsection{Proof of Theorem \ref{thm:main}}
%We first show that there exists a transcendental number \(\xi\) such that 
%\begin{eqnarray}
%\limsup_{n\to\infty}\|\xi\alpha^n\|=\frac{1}{\alpha+1}
%\label{yyy}.
%\end{eqnarray}
%Let \(\mathbf{t}=(t_m)_{m=-\infty}^{\infty}\) be defined as follows: 
%\begin{eqnarray*}
%t_{-m}:=\begin{cases}
%1 & (m=-n! \mbox{ for some }n\geq 4), \\
%-1 & (m=-1-n! \mbox{ for some }n\geq 4), \\
%0 & \mbox{(otherwise)}.
%\end{cases}
%\end{eqnarray*}
%Put
%\begin{eqnarray*}
%\xi:=\sum_{i=-\infty}^{\infty}\sum_{j=0}^{\infty} \alpha^i\alpha_2^j t_{i+j}=
%\sum_{j=0}^{\infty} \alpha^{-2j} \sum_{l=-\infty}^{\infty}\alpha^l t_l
%\end{eqnarray*}
%by \(\alpha_2=\alpha^{-1}\). It is well-known that \(\xi\) is transcendental (for instance, see \cite{nis}). 
%Since 
%\begin{eqnarray*}
%|(\ldots t_{-n+1} t_{-n} t_{-n-1}. t_{-n-2} t_{-n-3}\ldots)_{\alpha}|
%<
%\frac{\alpha}{\alpha^2-1}\sum_{h=\infty}^{\infty}\alpha^{-|h|}
%<\frac12,
%\end{eqnarray*}
%we get by (\ref{Eqn:aaa}) that 
%\begin{eqnarray}
%\varepsilon(\xi\alpha^n)
%=
%(\ldots t_{-n+1} t_{-n} t_{-n-1}. t_{-n-2} t_{-n-3}\ldots)_{\alpha}.
%\label{Eqn:16}
%\end{eqnarray}
%Using Lemma \ref{lex} and (\ref{Eqn:10}), we see that \(\xi\) satisfies (\ref{yyy}). \par
Assume that $\xi\in \R$ satisfies
\begin{eqnarray}
\eta=\limsup_{n\to\infty}\|\xi\alpha^n\|\le \frac{1}{\alpha+1}.
\label{eqn:1}
\end{eqnarray}
%and take a limsup word $(w_n)_{n\in \Z}$ for $\xi$. 
We show that $|s_{n}(\alpha;\xi)|\le 1$ for any sufficiently large $n$. 
%Let  \(\xi\not\in\mathbb{Q}(\alpha)\) be a real number satisfying (\ref{eqn:1}). 
%Note that (\ref{Eqn:3}) implies
%\begin{eqnarray}
%\varepsilon(\xi\alpha^n)
%=
%\frac{\alpha}{\alpha^2-1}
%\sum_{h=-\infty}^{\infty}
%\alpha^{-|h|}s_{h-n-1}(\alpha;x).
%\label{eqn:3}
%\end{eqnarray}
%We now show for any sufficiently large integer \(n\) that 
%\begin{eqnarray*}
%|s_{-n}(\alpha;\xi)|\leq 1.
%\end{eqnarray*}
In fact, suppose \(|s_{n}(\alpha;\xi)|\geq 2\) for infinitely many \(n\geq 0\). 
By (\ref{Eqn:2}), 
\begin{eqnarray*}
|\varepsilon( \xi  \alpha^{n+1})+b \varepsilon( \xi  \alpha^{n})
+\varepsilon( \xi \alpha^{n-1})|
\geq 2.
\end{eqnarray*}
Hence, we get 
$$
\max_{i\in\{n+1,n,n-1\}}\|\xi\alpha^i\|
\geq
\frac{2}{2+b}=
\frac{2}{2+\alpha+\alpha^{-1}}>
\frac{1}{\alpha+1},
$$
which contradicts (\ref{eqn:1}). In particular, any limsup word $(w_n)_{n\in \Z}$ for $\xi$ satisfies $|w_n|\le 1$ for all $n\in \Z$. 

In the sequel, we list forbidden subwords of $(w_n)$. %the limsup word $(w_n)$. 
We first show that the words $010$ and $11$ are forbidden. Suppose that $010$ or $11$ appears in $(w_n)$. 
Remark \ref{rem:limsup} implies that $010$ or $11$ appears infinitely many times in $(s_n)$. 
Observe that  
$$
0^{\infty}1.\overline{1}0^{\infty}\ll \dots s_{n-2}01.0 s_{n+2}\dots
$$
and
$$
0^{\infty}1.\overline{1}0^{\infty}\ll \dots s_{n-2}11.s_{n+1}\dots
$$
if $s_n\in \{-1,0,1\}$ for any $n\in \Z$. 
Thus, we conclude by (\ref{new1}) and Lemma \ref{lex} that $\limsup_{n\to\infty}\|\xi\alpha^n\|>1/(1+\alpha)$, which contradicts (\ref{eqn:1}). %for any $s_n\in \{-1,0,1\}$, we see that the words $010$ and $11$ are forbidden. }
From $g((-s_n))=-g((s_n))$, if a subword $v_1v_2\dots v_{\ell}$ is forbidden in $(w_n)$, then $(-v_1)\dots (-v_{\ell})$
also does not show in $(w_n)$. Therefore $0\overline{1}0$, $\overline{1}\hspace{0.4mm}\overline{1}$ are forbidden as well. 
Hereafter we skip this symmetric discussion within the proof of Theorem \ref{thm:main}.

%\(\sigma^{R-1}(\mathbf{w}_{-})\). \par
%Let \(\mathbf{y}=(y_n)_{n=-\infty}^{-1}\) be a sequence. 
%For any 
%\(l\geq 1\), we call \(y_{-1}y_{-2}\ldots y_{-l}\) an prefix of \(\mathbf{y}\). 
%Put 
%\begin{eqnarray*}
%{\mathcal{S}}:=
%\end{eqnarray*}
\begin{lem}
$
10^k1$ and $\overline{1} 0^k \overline{1}$
are forbidden for $k\ge 0$.
\end{lem}
\begin{proof}
This is valid for $k=0$. Assume that they are forbidden for $k\le t-1$.
If $10^{t}1$ appeared in $(w_n)$, then since
$$
\dots 10^{t}1.\overline{1}s_{n-2}\ldots\gg \ 0^{\infty}1.\overline{1}0^{\infty}
$$
is not possible by Lemma \ref{lex}, from the induction assumption
we see $s_{n-k}=0$ for $k=2,\dots, t+1$, arriving at a contradiction.
%The same works for $\overline{1}0^t\overline{1}$. 
Therefore the words are forbidden for $k=t$ as well.
\end{proof}

\begin{lem}
$0(1\overline{1})^k10$ and $0(\overline{1}1)^k\overline{1}0$ are forbidden
for $k\ge 0$.
\label{lem:3}
\end{lem}
\begin{proof}
We already know the case $k=0$. Assume that these words are forbidden for $k\le t-1$
and $0(1\overline{1})^t10$ appeared in the limsup word $(w_n)$, Then since
$$
\dots 0(1\overline{1})^t1.0s_{n-2}\ldots\gg \ 0^{\infty}1.\overline{1}0^{\infty}
$$
is not allowed, we obtain $s_{n-2}=\overline{1}$. 
Since $0\overline{1}0$, $\overline{1}\hspace{0.4mm}\overline{1}$ are
forbidden, this implies $s_{n-3}=1$. Continuing in this manner, we must have 
$s_{n-2}\dots s_{n-2t-1}=(\overline{1}1)^{t}$, arriving at a contradiction.
%The same works for $0(\overline{1}1)^t\overline{1}0$. 
\end{proof}
Having these forbidden words, we see either $(w_n)\in (0^*+(1\overline{1})^*)^{\Z}$
or  $(w_n)\in (0^*+(\overline{1}1)^*)^{\Z}$ holds.  If $(w_n)\in (0^*+(1\overline{1})^*)^{\Z}$,
then by using the monoid morphism \(\gamma:\A^{\ast}\to\B^{\ast}\) defined by 
\begin{eqnarray*}
\gamma(1)=1\overline{1}
\mbox{, }
\gamma(0)=0,
\end{eqnarray*}
we obtain 
$(w_n)=\gamma((x_n))$ with some $(x_n)\in \A^{\Z}.$
When $(w_n)\in (0^*+(\overline{1}1)^*)^{\Z}$,  we use
\begin{eqnarray*}
\widetilde{\gamma}(1)=\overline{1}1
\mbox{, }
\widetilde{\gamma}(0)=0,
\end{eqnarray*}
to get $(w_n)=\widetilde{\gamma}((x_n))$ with some $(x_n)\in \A^{\Z}$.  Now we study forbidden words of $(x_n)$.

\begin{lem}
\label{localbalance}
$F=\{ 0v01\tilde{v}1, 1\tilde{v}10v0\ |\ v\in \A^* \}$
is a set of forbidden words of $(x_n)_{n\in \Z}$. 
\end{lem}

\begin{proof}
From $g((s_n))=g((s_{-n}))$, it is enough to show that $0v01\tilde{v}1$ is forbidden.
%Assume that \(1\widetilde{v}10v0\) is a subword of the limsup word $(w_n)$. 
%Then the word 
%\begin{eqnarray*}
%1\overline{1}\gamma(\widetilde{v}) 1\overline{1} 0 \gamma(v)0
%\end{eqnarray*}
%occurs infinitely many times in \(\mathbf{w}_{-}\). Thus, there exist infinitely many 
%\(n\geq 0\) such that 
%\begin{eqnarray*}
%\|\xi'\alpha^n\|
%=
%|(\ldots
%1
%\underbrace{\overline{1}}_{\rho}
%\gamma(\widetilde{v}) 1\overline{1}. 0 \gamma(v)
%\underbrace{0}_{-\rho} 
%\ldots)_{\alpha}|,
%\end{eqnarray*}
%where \(\rho=2+|\gamma(\widetilde{v})|\).
%Note that 
%\begin{eqnarray*}
%|(
%\underbrace{\overline{1}}_{\rho}
%\gamma(\widetilde{v}) 1\overline{1}. 0 \gamma(v)
%\underbrace{0}_{-\rho} 
%)_{\alpha}|
%&=
%|(\underbrace{\overline{1}}_{\rho} 0^{\rho-2}1\overline{1}.0^{\rho-1}\underbrace{0}_{-\rho})_{\alpha}|\\
%&=
%(\underbrace{1}_{\rho}0^{\rho-2}\overline{1}1.0^{\rho-1}\underbrace{0}_{-\rho})_{\alpha}.
%=
%\frac1{\alpha+1}+\frac{\alpha}{\alpha^2-1}\alpha^{-\rho}.
%\end{eqnarray*}
%Hence, we obtain by Lemma \ref{lex} that 
%\begin{align*}
%\limsup_{n\to\infty}\|\xi'\alpha^n\| 
%&\geq
%(\overline{1}^{\infty}\underbrace{1}_{\rho}0^{\rho-2}\overline{1}1.
%0^{\rho-1}\underbrace{0}_{-\rho}\overline{1}^{\infty})_{\alpha}\\
%&>(\overline{1}1.0)_{\alpha}=\frac{1}{\alpha+1},
%\end{align*}
%which contradicts (\ref{XXXP}). \par
Assume that \(0\widetilde{v}01v1\) is a subword of the limsup word  $(x_n)$. 
Then the word 
\begin{eqnarray*}
0\gamma(\widetilde{v}) 0 1 \overline{1}  \gamma(v)1\overline{1}
\end{eqnarray*}
appears in $(w_n)$. 
Then
\begin{eqnarray*}
|(0\gamma(\widetilde{v}) 0 1. \overline{1}  \gamma(v)
1)_{\alpha}|=
(00^{\rho+1}1.\overline{1}0^{\rho}1)_{\alpha},
\end{eqnarray*}
where $\rho=|\gamma(\widetilde{v})|$.
Hence, 
$$
(\overline{1}^{\infty}0^{\rho+1}01.\overline{1}0^{\rho}1\overline{1}^{\infty})_{\alpha}
>(01.\overline{1})_{\alpha}=\frac{1}{\alpha+1},
$$
a contradiction. 
\end{proof}

In summary, if $\eta=\limsup_{n\to \infty} \| \xi \alpha^n\|\le 1/(1+\alpha)$, then 
the limsup word $(w_n)\in \B^{\Z}$ corresponding to $\xi$
must have a preimage $\x=(x_n)\in \A^{\Z}$ by $\gamma$ or $\tilde{\gamma}$.
Moreover, using the proof of Lemma \ref{localbalance}, we see that if %the strict inequality
$\eta<1/(1+\alpha)$ then there exists  $m\in \N$ satisfying
$\sigma^m((s_n)_{n\in \N})\in \gamma(\Y_F)$.

By Theorem \ref{Balanced}, $(x_n)$ is balanced. We say that a balanced word $(x_n)$ is 
{\bf symmetric}, if it is a mechanical word of intercept $0$, i.e., 
$x_0x_1\in \{01,10\}$ and $x_{n}=x_{-n+1}$ for $n\ge 2$. 
Then we have a

\begin{lem}
\label{Small}
If $(x_n)$ is balanced, then 
$|g(\gamma((x_n))|\le 1/(1+\alpha)$ and $|g(\tilde{\gamma}((x_n))|\le 1/(1+\alpha)$
holds. The equality holds if and only if $(x_n)$ is symmetric.
\end{lem}

\begin{proof}
We only show the case of $\gamma$.
Since $(w_n)=\gamma((x_n))\in (0^*+(1\overline{1})^*)^{\Z}$, unless $w_{-1}w_0w_1=01\overline{1}$,
we have $g((w_n))<1/(1+\alpha)$. If $w_{-n}+w_n=0$ for all $n>1$, then $g((w_n))=1/(1+\alpha)$.
Otherwise, there exists $n\in \N$ that $w_{-n}+w_n\neq 0$. Then we can find $v\in \A^*$ that 
$\gamma(v01\tilde{v})=w_{-m}\dots w_{-2}01.\overline{1} w_2\dots w_m$ with $w_{n}+w_{-n}=0$ 
for
$n\le m$ but $w_{m+1}+w_{-m-1}\neq 0$. 
Since $(x_n)\in \X_F$, $\gamma(0v01\tilde{v}1)$ does not 
show, it must be the image of $\gamma(1v01\tilde{v}0)$. Therefore $w_{-m-2}=-1$ and $w_{m+2}=0$
gives $g((w_n))<1/(1+\alpha)$.%\frac{1}{1+\alpha}$. 
\end{proof}
The proof above also shows that if $\iota(\x)<\infty$, then $$
\sup_{n\in \Z}|g(\sigma^n(w_n))|<\frac 1{1+\alpha}.
$$
We claim that $\eta<1/(1+\alpha)$ implies $\iota(\x)<\infty$. 
In fact, if $\eta<1/(1+\alpha)$ and $\iota(\x)=\infty$, then we may assume that for any $n\in \N$, there
is a word $v\in \A^+$ with $|v|=n$ that $\tilde{v}01v$ 
is a factor of $\x$ (the proof is similar for $\tilde{v}10v$).
This would imply 
$\gamma(\tilde{v})01\overline{1}\gamma(v)$ (or
$\widetilde{\gamma}(\tilde{v})0\overline{1}1\widetilde{\gamma}(v)$)
is a factor of the limsup word $(w_n)$. 
If $\gamma(\tilde{v})01\overline{1}\gamma(v)$ is a factor, then
\begin{equation}
\label{Approach}
\left|(\dots\gamma(\tilde{v})01.\overline{1}\gamma(v)\dots)_\alpha - \frac 1{1+\alpha}\right|\le 
\frac {2\alpha^{-n-2}}{1-\alpha},
\end{equation}
which yields a contradiction if $n$ is sufficiently large. %we get a contradiction if we take $n$ sufficiently large.
The case $\widetilde{\gamma}(\tilde{v})0\overline{1}1\widetilde{\gamma}(v)$ is similar.
Thus, we proved the claim. %This shows the claim.
By Proposition \ref{Iota} we see $\eta<1/(1+\alpha)$ implies $(x_n)$ is purely periodic.
Therefore there exists a Christoffel word $\nu\in \A^*$ that $(x_n)\simeq \nu^{\Z}$ 
and $(w_n)\simeq \gamma(\nu)^{\Z}$.
Here $\x\simeq \y$ means there exists $n\in \Z$ that $\x=\sigma^n(\y)$.
Since $\sigma^m((s_n)_{n\in \N})\in \gamma(\Y_F)$ for some $m\in \N$ 
when $\eta<1/(1+\alpha)$, this implies that $(s_n)$ and $(w_n)$
share their tails up to some shift, i.e., there exist $m_1,m_2\in \N$ that 
$\sigma^{m_1}((s_n)_{n\in \N})=\sigma^{m_2}((w_n)_{n\in \N})$.
Consequently we have $$\limsup_{n\to \infty} \|\xi \alpha^n\|=
\limsup_{n\to \infty} |g(\sigma^n((w_n))|.$$
Therefore $\LL(\alpha)\cap [0,1/(1+\alpha))$ is completely described by Christoffel words.
By Lemma \ref{Maxiota}, we evaluate
the value $g((w_n))$ for $(w_n)\simeq \gamma(\nu^{\Z})$ that the letter at index $0$ of 
$\x=\nu^{\Z}$ is the last letter of the Christoffel word $\nu$, 
and $w_{-1}w_0w_1=01\overline{1}$ (or $1\overline{1}0$) for $|\nu|\ge 2$,
which is enough to obtain $\limsup_{n} |g(\sigma^n((w_n)))|$. 
In fact, the computation below shows that the limit superior depends only on $\iota(\x)$
because any central word is a palindrome.
If $|\nu|=1$, then 
$$
|(0^{\Z})_{\alpha}|=0
$$
and
$$
|((1\overline{1})^{\Z})_{\alpha}|=\frac{1}{\alpha-\alpha_2} \left(
1+\frac{-2\alpha^{-1}+2\alpha^{-2}}{1-\alpha^{-2}} \right)
=\frac 1{b+2}.
$$
By recalling 
$s_{-n}(\alpha;\xi)=0$ for sufficiently large $n$, the tail $0^{\N}$ occurs if and only if
$(s_n)$ ends up $0^{\infty}$ in both directions, i.e., 
$$
\xi\in X_0:=\frac{1}{\alpha-\alpha_2} \Z[\alpha]
$$
and tails $(1\overline{1})^{\N}$ or $(\overline{1}1)^{\N}$ occur if and only if
$$
\xi\in X_1:=\pm \frac{1}{\alpha-\alpha_2} \left(\frac{1-\alpha^{-1}}{1-\alpha^{-2}}+\Z[\alpha] \right).
$$
If $|\nu|>1$, then by Lemma \ref{Central}, there exists a palindrome $v\in \A^*$ that $\nu=0v1$ or $1v0$. 
Consider the case $\nu=1v0$.
Since $\gamma((1v0)^{\Z})=(1\overline{1}\gamma(v)0)^{\Z}$, we have
\begin{eqnarray*}
&&((01\overline{1}\gamma(v))^{\infty}01.\overline{1}(\gamma(v)01\overline{1})^{\infty})_{\alpha}\\
&=& \frac 1{\alpha-\alpha_2}\left( 1-\frac 1{\alpha} + \frac{-\alpha^{-n}+2\alpha^{-n-1}-\alpha^{-n-2}}{1-\alpha^{-n-1}}\right)=:z_n <(0^{\infty}1.\overline{1}0^{\infty})_{\alpha}
\end{eqnarray*}
for $|\gamma(v)|+2=n$. By Lemma \ref{Maxiota}, we see $\iota((1v0)^{\Z})=|v|$
and $\limsup_m \|\xi \alpha^m\|$ attains the value $z_n<1/(1+\alpha)$. 
Since
$$
\frac{1}{z_{n}}-1=\frac {\alpha^{n+2}-1}{\alpha(\alpha^{n}-1)}
$$
is invariant under the Galois conjugation $\alpha\mapsto \alpha_2$, we see $z_n\in \Q$. 
Moreover, using
\begin{equation}
\label{NegCF}
b-\frac {\alpha^{n+2}-1}{\alpha(\alpha^{n}-1)}= \frac{\alpha(\alpha^{n-2}-1)}{\alpha^{n}-1},
\end{equation}
we get that
$$
\frac{1}{z_{2n}}-1=
\frac {\alpha^{2n+2}-1}{\alpha(\alpha^{2n}-1)}=
b-\cfrac{1}{b-\cfrac {1}{b-\cfrac{1}{\ddots}}}=b-[0;\underbrace{b,b,\dots,b}_{n-1}]_{neg},%[\underbrace{b,b,\dots,b}_n]_{neg},
$$
which proves $z_{2n}=p_{2n}/q_{2n}$ by the statement after Theorem \ref{thm:main}.
Since (\ref{NegCF}) holds for odd $n$ as well,  the proof for $z_{2n-1}$ is similar.
Switching to the case $\nu=0v1$ or the case using the map $\tilde{\gamma}$, we obtain
the same values of $\limsup_m \|\xi \alpha^m\|$ because the absolute value of $g$ does not change. 
By this proof, we see for $n\ge 2$ that
$$
X_n=\left\{ \xi\in \R\ \left|\ \limsup_{m\to \infty} \| \xi \alpha^m\|=z_n \right.\right\}
$$
corresponds to purely periodic balanced words generated by
Christoffel words $\nu$ whose central word $v$ satisfies
$|\gamma(v)|+2=n$. 
This formula is consistent with $n=0,1$ as well.
More explicitly we have
\begin{equation}
\label{Xn}
X_n = \left\{\left. \pm g\left(\sigma^k(\x)\right)+ 
\frac{1}{\alpha-\alpha_2} \Z[\alpha] \ \right|\ \begin{aligned}
&\x=0^{\infty}1.\overline{1}(\gamma(v)01\overline{1})^{\infty},\ 
0\le k\le n \\
&|\gamma(v)|+2=n,\  v:\text{central} \end{aligned}
\right\}
\end{equation}
for $n\ge 2$.
By Lemma \ref{Small}, $X_{\infty}$ contains
$$
\mathcal{Q}:=
\left\{\left. \pm g\left(0^{\infty}.s_{k}s_{k+1}\dots \right) + 
\frac{1}{\alpha-\alpha_2} \Z[\alpha] \ \right|\ 
\begin{aligned}
&(s_n)_{n\in \N}=\gamma(\y),\  k=1,2\\
& \y:\text{sturmian}
\end{aligned}
\right\}.
$$
For a given element of $\mathcal{Q}$, the slope of its 
corresponding sturmian word is uniquely retrieved,
we see $X_{\infty}$ is uncountable.
We claim that $X_{\infty}$ corresponds to 
the set of all eventually balanced aperiodic words, i.e.,
$$
X_{\infty}=
\left\{\pm g\left(0^{\infty}. s_{k}s_{k+1}\dots \right) + \frac{1}{\alpha-\alpha_2} \Z[\alpha]
 \ \left|\ 
\begin{aligned}
&(s_n)_{n\in \N}=\gamma(\y), \ k=1,2\\ 
&\y : \text{not eventually periodic}\\
&\quad \text{ and eventually balanced}
\end{aligned}
\right.\right\}.$$
In fact, if $\xi\in X_{\infty}$, then its corresponding bi-infinite words $(x_n)$
are balanced but may have different slopes, c.f. \cite[Proposition 2]{Yasutomi}. 
However for any $\varepsilon>0$, there exists $m_0$ that if $m\ge m_0$ then
$\|\xi \alpha^m\|<1/(1+\alpha)+\varepsilon$. 
By refining the proof of this section, 
the length of forbidden words in $F$ we may observe in 
$\y=(y_n)$ for $n\ge m$
diverges as $m\to \infty$.
By quantifying Theorem \ref{Balanced}, the minimum length of a word containing an unbalanced pair $0w0$ and $1w1$
in Lemma \ref{Unbalance} must be large if the minimum length of forbidden words in $F$ is large.
Thus the claim follows from (\ref{Approach}).

Since being eventually balanced is a tail event, i.e., invariant by changing a finite number of terms, 
there are many eventually balanced words which is not balanced. We now give two further examples. %Moreover we have

\begin{ex}
Let $(n_i)_{i\in \N}$ be an integer sequence with $\lim_{i\to \infty} n_i=\infty$.
A one-sided infinite word
$$
0^{n_1}10^{n_2}10^{n_3}1\dots
$$
is eventually balanced but not balanced.
\end{ex}

%and

\begin{ex}
Let $\tau$ be the substitution defined by: $\tau(0)=01,\tau(1)=0$. Then the unique fixed point $\tau(w)=w\in \{0,1\}^{\N}$ is
the Fibonacci word, the most famous sturmian word. We consider the one-sided infinite word
$$
x=\tau(0)\tau(0)\tau^2(0)\tau^2(0)\tau^3(0)\tau^3(0)\dots.
$$
Then for any $m\in \N$, $\sigma^m(x)$ is eventually balanced but not balanced. 
In fact, 
$\tau^n(0)\tau^n(0)=\tau^n(00)$, $\tau^n(0)\tau^{n+1}(0)=\tau^n(001)$
is a factor of $w$ and therefore balanced, which implies that $x$ is eventually balanced. 
On the other hand, by considering the frequency of letters, $x$ can not be periodic. 
If $\sigma^m(x)$ were balanced, then $\sigma^m(x)$ is a sturmian word. 
However, since $\tau$ is a sturmian morphism (see \cite{Lothaire:02}),
$\tau^n(0)\tau^{n}(0)\tau^{n+1}(0)\tau^{n+1}(0)\dots$ is sturmian implies $00x$ is balanced
but has a forbidden prefix $000101\in F$.
\end{ex}

Therefore $\mathcal{Q}$ is a proper subset of $X_{\infty}$ and the difference 
$X_{\infty}\setminus \mathcal{Q}$ is uncountable.
Moreover there are eventually balanced words whose frequency of $1$ does not converge, see 
\cite[Proposition 2]{Yasutomi}.

%Note that $(0011)^{\infty}$ is an example of non-eventually balanced words. 

%Recall that $K$ consists of four elements $0^{\Z}, 1^{\Z}, 0^{\infty}10^{\infty}$ and $1^{\infty}01^{\Z}$.
%Since 
%\begin{eqnarray*}
%|(( 0^{\infty}1.\overline{1}0^{\infty}))_{\alpha}|=
%|(( (1\overline{1})^{\infty}01.\overline{1}(1\overline{1})^{\infty}))_{\alpha}|=\frac 1{1+\alpha}
%\end{eqnarray*}
%we have $(x_n)=0^{\Z}$ or $(x_n)=1^{\Z}$, and the limsup word $(w_n)$ must be 
%$0^{\Z}$ or $(1\overline{1})^{\Z}$. Thus we obtain only two possibility of values
%of $\eta<1/(1+\alpha)$, that is, 
%Then $(s_n(\alpha,\xi))_{n\in \Z}$ end up with $0^{\N}$ or $(1\overline{1})^{\N}$, 
%since otherwise we can select a limsup word different from $0^{\Z}$ and $(1\overline{1})^{\Z}$, which causes
%a contradiction. 

%We are left to describe the set $X_3$ which corresponds to $\eta=1/(1+\alpha)$.
%Let $(x_n)$ be the preimage of a limsup word $(w_n)$ for $\xi$ that 
%$\limsup_{n\to \infty} \|\xi \alpha^n\|=1/(1+\alpha)$. Then we have proved that $(x_n)\in \X_F$. 

\subsection{Proof of Theorem \ref{thm:main2}}
We prove Theorem \ref{thm:main2} in a similar way as in the proof of Theorem \ref{thm:main}.
Our idea is to reduce the problem on $\limsup_{n\to \infty} \|\xi \alpha^n\|$ to
the one on $\limsup_{n\to \infty} \|\xi \alpha^{2n}\|$ and obtain the conclusion because 
$0<\alpha_2^2<1$.
%We first show that there exists a transcendental number \(\xi\) such that 
%\begin{align}\label{eqn:2-1}
%\limsup_{n\to\infty}\|\xi\alpha^n\|=\frac{\alpha^2-1}{\alpha(\alpha^2+1)}. 
%\end{align}
%Let \(\mathbf{t}=(t_m)_{m=-\infty}^{\infty}\) be defined by 
%\begin{eqnarray*}
%t_{-m}:=\begin{cases}
%1 & (m=-n! \mbox{ for some }n\geq 4), \\
%-1 & (m=-2-n! \mbox{ for some }n\geq 4), \\
%0 & \mbox{(otherwise)}.
%\end{cases}
%\end{eqnarray*}
%and let 
%\begin{eqnarray*}
%\xi:=\sum_{i=-\infty}^{\infty}\sum_{j=0}^{\infty} \alpha^i\alpha_2^j t_{i+j}=
%\sum_{j=0}^{\infty} (-1)^j\alpha^{-2j} \sum_{l=-\infty}^{\infty}\alpha^l t_l
%\end{eqnarray*}
%by $\alpha_2=-\alpha^{-1}$. Thus, we see that  the transcendental number $\xi$ satisfies (\ref{eqn:2-1}). \par
Assume that there exists $\xi\in \R$ satisfying 
\begin{align*}
\limsup_{n\to\infty}\|\xi\alpha^n\|<\frac{b}{\alpha^2+1}=(1.0\overline{1})_{\alpha}, 
\end{align*}
and select a limsup word $(w_n)$ for $\xi$.
%we construct a sequence $\mathbf{x}\in \mathcal{W}$, which contradicts Lemma \ref{lem:w}. 
%In the same way as in the proof of Theorem \ref{thm:main}, we can show that there exists a positive integer $R$ 
%satisfying 
%\[|s_{-n}(\alpha;\xi)|\leq 1\]
%for any $n\geq R$. Let $\mathbf{w}=(w_n)_{n=-\infty}^{\infty}$ and $\mathbf{w}_{-}=(w_n)_{n=-1}^{\infty}$ be defi%ned as 
%follows: 
%\begin{align}\label{eqn:2-2}
%w_n:=\left\{
%\begin{array}{cc}
%s_n(\alpha;\xi) & (n\leq -R),\\
%0 & (n>-R).
%\end{array}
%\right.
%\end{align}
%Recall that $\mathbf{w}_{-}$ is not eventually periodic. 
%Let again $\sigma$ be the shift operator. 
%Put
%\begin{eqnarray*}
%\xi':=\sum_{i=-\infty}^{\infty}\sum_{j=0}^{\infty} \alpha^i\alpha_2^j w_{i+j}=
%\sum_{j=0}^{\infty} (-1)^j\alpha^{-2j} \sum_{l=-\infty}^{\infty}\alpha^l w_l.
%\end{eqnarray*}
%In the same way as in the proof of (\ref{Eqn:16}), we get by 
%\[
%\frac{\alpha}{\alpha^2+1}\sum_{h=-\infty}^{\infty}
%\alpha^{-|h|}<\frac12
%\]
%that 
%\begin{align}\label{eqn:2-3}
%\varepsilon(\xi'\alpha^n)&=(\ldots w_{-n+1} w_{-n} w_{-n-1}. w_{-n-2} w_{-n-3}\ldots)_{\alpha}\nonumber\\
%&=\frac{\alpha}{\alpha^2+1}\sum_{h=-\infty}^{\infty}
%(-1)^{\max\{0,h\}}
%\alpha^{-|h|}w_{h-n-1}
%\end{align}
%for each integer \(n\). Moreover, we get
%\begin{align}
%\limsup_{n\to\infty}\|\xi'\alpha^n\|&=\limsup_{n\to\infty}\|\xi\alpha^n\|\nonumber\\
%&<\frac{\alpha^2-1}{\alpha(\alpha^2+1)}=(01.0\overline{1})_{\alpha}.\label{eqn:2-4}
%\end{align}
In what follows, we find forbidden subwords of $(w_n)$. 
%If necessary, we change $R$ in (\ref{eqn:2-2}).
%Using Lemma \ref{lex} and (\ref{eqn:2-4}), 
We easily see that the words in 
%$\mathcal{S}_1$ appear at most 
%finitely many times in $\mathbf{w}_{-}$: 
\[\mathcal{S}_1=\{\overline{1}11, 011, \overline{1}10, 1\overline{1}\hspace{0.4mm}\overline{1}, 
0\overline{1}\hspace{0.4mm}\overline{1}, 1 \overline{1}0\}.\]
are forbidden. 
For instance, we see $\dots s_{n-2}\overline{1}1.0s_{n+1}\ldots \gg 0^{\infty}1.0\overline{1}0^{\infty}$
for any $s_n\in \{-1,0,1\}$, which shows $\overline{1}10$ does not appear in $(w_n)$.
For brevity, we write this reasoning as
$$
\overline{1}1.0\gg 0^{\infty}1.0\overline{1}0^{\infty}.
$$
%Putting
%\[\left(\underbrace{(1\overline{1})^{\infty}}\overline{1}1.0\overline{1}^{\infty}\right)_{\alpha}:=
%(\underbrace{\ldots 1\overline{1}1\overline{1}1\overline{1}}\hspace{0.4mm}\overline{1}1.0\overline{1}^{\infty})_{\alpha},\]
%we get 
%\begin{align*}
%\limsup_{n\to\infty}\|\xi'\alpha^n\|\geq \left((1\overline{1})^{\infty}{\bf \overline{1}1.0}\overline{1}^{\infty}\right)_{\alpha}
%>(0^{\infty}{\bf 01.0}\overline{1}0^{\infty})_{\alpha},
%\end{align*}
%
%If necessary, changing \(\mathbf{w}_{-}\) by a suitable shift \(\sigma^{h}(\mathbf{w}_{-})\) (\(h>0\)), 
%we may assume that the words in $\mathcal{S}_1$ do not appear in \(\sigma^{R-1}(\mathbf{w}_{-})\). 
%In particular, since $\mathbf{w}_{-}$ is not eventually periodic and since $\overline{1}11$ and $1\overline{1}\overline{1}$ do 
%not appear in $\mathbf{w}_{-}$, we see that $0$ appears infinitely many times in $\mathbf{w}_{-}$. \par
Next, we verify that the words in 
%$\mathcal{S}_2$ appear at most 
%finitely many times in $\mathbf{w}_{-}$: 
\[\mathcal{S}_2=\{11, 1\overline{1}, \overline{1}1, \overline{1}\hspace{0.4mm}\overline{1}\}\]
are forbidden. %, because of $S_1$. 
In fact, the letter prepends to $11$ must be $1$ because any word in $\mathcal{S}_1$ is forbidden. Thus, $11$ is forbidden by  
$111.11\gg 0^{\infty}001.0\overline{1}0^{\infty}.$
%for any $s_n,t_n\in \A$. For brevity, we write this reasoning as
%$$
%111.11 \gg 001.0\overline{1}
%$$
Similarly $1\overline{1}$ must be followed by $1$, and $\overline{1}1$ must be followed by $\overline{1}$. 
Thus, each of $1\overline{1}$ and $\overline{1}1$ is forbidden by 
$1\overline{1}1.\overline{1}1 \gg 0^{\infty}001.0\overline{1}0^{\infty}.$
%If certain word  in $\mathcal{S}_2$ occurs infinitely many times in $\mathbf{w}_{-}$, then there exist $v \in S$ and 
%infinitely many positive integers $n$ such that $w_{-n} w_{-n-1}=v$ and $w_{-n-2}=0$. If $w_{-n-1}=1$, then we see 
%for infinitely many positive integers $n$ such that 
%\[w_{-n+3}w_{-n+2}w_{-n+1}w_{-n}w_{-n-1}w_{-n-2}=111110\]
%because $\overline{1}10, \overline{1}11, 011$ appear at most finitely many times in $\mathbf{w}_{-}$. 
%Thus, we get 
%\begin{align*}
%\limsup_{n\to\infty}\|\xi'\alpha^n\|\geq \left((1\overline{1})^{\infty}{\bf 111.11}\overline{1}^{\infty}\right)_{\alpha}
%>(0^{\infty}{\bf 001.0\overline{1}}0^{\infty})_{\alpha},
%\end{align*}
%a contradiction. Similarly, if $w_{-n-1}=\overline{1}$, then 
%\[w_{-n+3}w_{-n+2}w_{-n+1}w_{-n}w_{-n-1}w_{-n-2}=
%\overline{1}\hspace{0.4mm}\overline{1}\hspace{0.4mm}\overline{1}\hspace{0.4mm}\overline{1}\hspace{0.4mm}\overline{1}0
%\]
%for infinitely many positive integers $n$, a contradiction. \par
%We may assume that the words in $\mathcal{S}_2$ do not appear in $\sigma^{R-1}(\mathbf{w}_{-})$. 
%Moreover, we may assume that 
Thirdly we show that
\[\mathcal{S}_3:=\{101, \overline{1}0\overline{1}, 00100,00\overline{1}00\}\]
is a set of forbidden words because we have $01.01 \gg 0^{\infty}001.0\overline{1}0^{\infty}$
and $00100\gg 0^{\infty}001.0\overline{1}0^{\infty}$.
%Similarly, we can show that $\overline{1}0\overline{1}, 00100$, and $00\overline{1}00$ occur 
%at most finitely many times in $\mathbf{w}_{-}$. \par
%Without loss of generality, we may assume that 
%\(w_{-R}=1\). Thus, \(\sigma^{R-1}(\mathbf{w}_{-})\) is denoted as 
%\begin{eqnarray*}
%\sigma^{R-1} (\mathbf{w}_{-})=)
%v_0^{p_0} 0^{q_0} v_1^{p_1} 0^{q_1} v_2^{p_2} 0^{q_2} \ldots, 
%\end{eqnarray*}
%where, for any \(i\geq 0\), \(v_i\in\{10\overline{1}0,\overline{1}010\}\), 
%\begin{eqnarray*}
%p_i\in\frac12\mathbb{Z}^{+}, p_i\geq 1 
%\mbox{, and }
%q_i\in\mathbb{Z}^{+}.

%The words in the following set appear at most finitely many times in $\mathbf{w}_{-}$: 
Fourthly, we see
\[\mathcal{S}_4:=\{1001,100\overline{1}, \overline{1}001, \overline{1}00\overline{1}\}\]
is a set of forbidden words, which follows from
$0\overline{1}01.001\gg 0^{\infty}0\overline{1}01.0000^{\infty}$
and
$
\overline{1}001.0\overline{1}0 \gg
0^{\infty}0001.0\overline{1}00^{\infty}$. 
%Similarly, we can show that $100\overline{1}$ and $\overline{1}00\overline{1}$ are forbidden. 
%\end{proof}
%We may assume that any word in $\mathcal{S}_4$ does not appear in $\sigma^{R-1}(\mathbf{w}_{-})$. 
%In particular, we see $q_i\geq 2$ for any $i\geq 0$.
\begin{lem}\label{lem:2-1}
For any $k\geq 0$, the words $0(010\overline{1})^k0100$ and
$0(0\overline{1}01)^k0\overline{1}00$ are forbidden.
\end{lem}

\begin{proof} The case $k=0$ is in $S_3$. Assume that the statement is valid for $k\le n$
and $w=0(010\overline{1})^{n+1}0100$ is a factor of the limsup word. Since
$$
\dots 001.0\overline{1}(010\overline{1})^{n}0100\ldots \gg 0^{\infty}1.0\overline{1}0^{\infty}
$$
is not allowed, we see that $w$ must be inductively 
prepended by $s0\overline{1}(010\overline{1})^{n}$ with $s=0$ or $1$. 
Using the induction assumption, we get $s=1$. Thus, we obtain
$$
10\overline{1}(010\overline{1})^{n}001.0\overline{1}(010\overline{1})^{n}0100
\gg 0^{\infty}1.0\overline{1}0^{\infty},
$$
a contradiction.
\end{proof}

Let $\mathcal{S}_5$ be the set consisted of words of the form 
\begin{align*}
\begin{cases}
10^k1, \quad, \overline{1} 0^k \overline{1} & (k\geq 0), \\
\overline{1}0^{2l}1 , \quad 1 0^{2l} \hspace{0.2mm}\overline{1} & (l\geq 0).
\end{cases}
\end{align*}
\begin{lem}\label{lem:2-3}
$S_5$ is a set of forbidden words.
%There are at most finitely many \(h\geq 0\) such that 
%\(\sigma^{h}(\mathbf{w}_{-})\) has a prefix \(v_h\in{\mathcal{S}}_5\), where $a_h,b_h\in \{1,-1\}$ and $k_h\geq 3%$. 
\end{lem}

\begin{proof}
The statement is already shown for $k=0,1,2$ and $\ell=0,1$. Assume that
it is proved for $k, 2l \le n$. When $k=n+1$ is odd, since
\begin{equation}
\label{gg}
10^{k}1.0\overline{1}\ldots \gg 0^{\infty}0^{k+1}1.0\overline{1}0^{\infty}
\end{equation}
is not possible, $10^{k}1.0\overline{1}$ must be followed by $0^{k-2}\overline{1}$. 
Then we reach a forbidden word $\overline{1}0^{k-2}\overline{1}$ by induction assumption.
If $k=n+1$ is even, then the proof is simpler. 
$10^{k}1.0\overline{1}$ must be followed by $0^{k-2}s$ with $s\in \A$
but (\ref{gg}) holds regardless of the choice of $s$.
The same simpler reasoning applies to $\overline{1}0^{2l+2}1$ with $2l+2>n\ge 2l$.
\end{proof}

Considering all these forbidden words, we see that the limsup word $(w_n)$ belongs to
$((00)^*+(010\overline{1})^*)^{\Z}$ or 
$((00)^*+(0\overline{1}01)^*)^{\Z}$.
Thus we have $w_{2n-1}=0$ for $n\in \Z$ (or 
$w_{2n}=0$ for $n\in \Z$) and the problem is reduced to $\limsup_{n\to \infty} \| \xi \alpha^{2n}\|$.
The remaining proof is similar to Theorem \ref{thm:main}, by substituting $\alpha$ by $\alpha^2$. 

\section{$\LL(\alpha)$ contains an interval: quadratic unit case}
\label{Interval}
Define
the map $T$ from $[-1/2,1/2)$ to itself by
$$
T:x\mapsto \alpha x - \lfloor \alpha x+1/2\rfloor.
$$
Set $\D=(-(\alpha+1)/2,(\alpha+1)/2)\cap \Z$.
We define the 
coding map $d$ from $[-1/2,1/2)$ to $\D^{\N}$ by 
$d(x)=d_1d_2\dots\in \D^{\N}$, where $d_i=\lfloor \alpha T^{i-1}(x)+1/2\rfloor$. Note that 
$$x=\sum_{i=1}^{\infty} \frac{d_i}{\alpha^{i}}.$$
This expression is called {\bf symmetric beta expansion}
and studied in \cite{Akiyama-Scheicher:04}.
An infinite sequence $(d_i)_{i\in \N}\in \D^{\N}$ is realized as a symmetric beta 
expansion if and only if
$$
d\left(-\frac 12 \right)\lxe \sigma^k ((d_i)_{i\in \N}) \lx d\left(\frac 12\right)
$$
for all $k\in \{0\} \hspace{0.5mm} \cup \hspace{0.5mm} \N$ where $\lxe$ and $\lx$ are the natural lexicographical orders. 
Put $c=\lfloor (\alpha+1)/2 \rfloor$.   
It is useful to extend the domain of $T$ to $[-(c+1/2)/\alpha, (c+1/2)/\alpha)$. We obtain
an expansion in the same digits $\D$. 
If $x\in [1/2, (c+1/2)/\alpha)$, 
then the first digit is $c$ and $T(x)=\alpha x-c \in [-1/2,1/2)$ and 
if $x\in [-(c+1/2)/\alpha,-1/2]$, 
then the first digit is $\overline{c}$ and $T(x)=\alpha x+c \in [-1/2,1/2)$. 
Therefore its orbit falls into the original domain $[-1/2,1/2)$ after a single application of $T$.
The domain of the coding map $d$ is naturally extended to $[-(c+1/2)/\alpha, (c+1/2)/\alpha)$.

From now on, we restrict ourselves to the case when $\alpha$ is a quadratic unit, i.e., 
$\alpha=(b+\sqrt{b^2\mp 4})/2$. Direct
computation gives

\begin{lem}
For $3\le b\in \Z$ and $\alpha=(b+\sqrt{b^2-4})/2$, we have
\label{Adm}
$$
d\left(\frac12 \right)=c\,0(\overline{c}\,1)^{\infty},\quad
d\left(-\frac12\right)=(\overline{c}\,1)^{\infty}$$
when $b$ is even and 
$$d\left(\frac 12\right)=c\,c\,0(\overline{c}\,\overline{c}\,1)^{\infty},\quad 
d\left(-\frac12\right)=(\overline{c}\,\overline{c}\,1)^{\infty}$$
when $b$ is odd.
\end{lem}
%and
%
%\begin{lem}
%\label{Adm2}
%$$d\left(\frac{\alpha-\alpha_2}{2\alpha}\right)=c\,\overline{1}\,0^{\infty},\quad
%d\left(-\frac{\alpha-\alpha_2}{2\alpha}\right)=\overline{c}\,1\,0^{\infty},$$
%when $b$ is even and 
%$$d\left(\frac{\alpha-\alpha_2}{2\alpha}\right)=c\,(c-1)\,c\,0\,(\overline{c}\,\overline{c}\,1)^{\infty},\quad
%d\left(-\frac{\alpha-\alpha_2}{2\alpha}\right)
%=\overline{c}\,\overline{(c-1)}\,\overline{c}\,1\,(\overline{c}\,\overline{c}\,1)^{\infty}$$
%when $b$ is odd.
%\end{lem}

\begin{lem}
For $1\le b\in \Z$ and $\alpha=(b+\sqrt{b^2+4})/2$,  we have
\label{Adm3}
$$
d\left(\frac12 \right)=c\,1(\overline{c}\,0)^{\infty},\quad
d\left(-\frac12\right)=(\overline{c}\,0)^{\infty}$$
when $b$ is even and 
$$d\left(\frac 12\right)=c\,\overline{(c-1)}\,0(\overline{c}\,(c-1)\,1)^{\infty},\quad 
d\left(-\frac12\right)=(\overline{c}\,(c-1)\,1)^{\infty}$$
when $b$ is odd.
\end{lem}
%and
%
%\begin{lem}
%\label{Adm4}
%$$d\left(\frac{\alpha-\alpha_2}{2\alpha}\right)=c\,1\,0^{\infty},\quad
%d\left(-\frac{\alpha-\alpha_2}{2\alpha}\right)=\overline{c}\,\overline{1}\,0^{\infty},$$
%when $b$ is even and 
%$$d\left(\frac{\alpha-\alpha_2}{2\alpha}\right)=c\,\overline{(c-1)}\,(c-1)\,(1\,\overline{c}\,(c-1))^{\infty},\%quad
%d\left(-\frac{\alpha-\alpha_2}{2\alpha}\right)
%=\overline{c}\,(c-1)\,\overline{(c-1)}\,0\,(\overline{c}\,(c-1)\,1)^{\infty}$$
%when $b$ is odd.
%\end{lem}

Using these facts, we obtain 

\begin{theorem}
\label{Int}
If $\alpha\ge 3$ then there exists $\kappa<1/2$ that $[\kappa,1/2] \subset \LL(\alpha)$. 
In particular, $\LL(\alpha)$ has a positive Lebesgue measure. 
%$$\left[\kappa,\frac 12\right] \subset \LL(\alpha).$$
\end{theorem}

%Note that $\alpha\ge 3$ is equivalent to $c\ge 2$.

\begin{proof}
First we consider the case $\alpha=(b+\sqrt{b^2-4})/2$ and $\alpha_2=1/\alpha$. 
Then $\alpha\ge 3$ if and only if $b\ge 4$.
We will show that $\kappa=c/(\alpha+1)$ suffices for $b\ge 8$. Since $\alpha$ is irrational, we
have $\kappa<1/2$. 
Take $\eta\in [\kappa,1/2]$ and set %$\eta\in \left[\kappa,\frac 12\right]$ and set
%Since
%$$
%\kappa=\frac {1}{\alpha-\alpha_2} \left(c - \frac c{\alpha}\right),
%$$
%and Lemma \ref{Adm}, we have 
$$
d\left(\frac{\eta(\alpha-\alpha_2)}{\alpha}\right)=y_0y_1y_2 \dots.$$
% where $-c \le x_1 \le -1$
%for even $b$ and $-c \le x_1 \le c-1$ for odd $b$.
Define $(s_n)_{n\in \Z}$ by $s_0=y_0$, $s_n=\lceil y_n/2 \rceil$ and $s_{-n}=\lfloor y_n/2 \rfloor$
for $n\in \N$. 
Since $\lceil y/2 \rceil +\lfloor y/2 \rfloor =y$ for any $y\in \Z$, 
we see $((s_n))_{\alpha}=\eta$ from the definition.
We claim that $|g(\sigma^k((s_n)))|\le \eta$ for $k\in \Z$. In fact, since $s_n\le \lceil c/2 \rceil$
holds except $n=0$, it suffices to show an inequality
\begin{equation}
\label{Key}
\left\lceil \frac c2 \right\rceil \left( 1+ \sum_{i=1}^{\infty} \frac 2{\alpha^i} +\frac 1{\alpha} \right)= 
\left\lceil \frac c2 \right\rceil \left( 1+ \frac 2{\alpha-1} +\frac 1{\alpha} \right) 
\le c -\frac c{\alpha}.
\end{equation}
Note that the term $1/\alpha$ in the left side of (\ref{Key})
gives the maximum possible contribution
from the exceptional digit $c$ after shifting the sequence $(s_n)$.
Thus, (\ref{Key}) holds for $b\ge 8$, which shows the claim.
We expect $(s_n)$ to play a role of the limsup word. Take the central block 
$t(n)=s_{-n}\dots s_n$ and a sufficiently large integer $\ell$ %integer $\ell\gg 1$ 
and construct 
a word $(x_n)_{n\in \N}$ by $x_n=0$ for $n\le 0$ and 
$$
x_1x_2\dots= t(\ell)t(\ell+1)\dots.
$$
Then by (\ref{Key}), the conjunction part of $t(k)$ and $t(k+1)$ does no harm and we have
$$
\limsup_{k\to \infty} |g(\sigma^k((x_n)))|=\eta.
$$
This finishes the case $\alpha=(b+\sqrt{b^2-4})/2$ with $b\ge 8$.
For $b=5$ and $b=7$, we can confirm 
\begin{equation}
\label{Key2}
\left\lceil \frac c2 \right\rceil \left( 1+ \frac 2{\alpha-1} +\frac 1{\alpha} \right) 
\le c
\end{equation}
and the statement is valid for $$
\kappa=\frac {c}{\alpha-\alpha_2}<\frac 12.$$
% because
%the first digit of the expansion of $\eta\in [\kappa, 1/2)$ is $c$.
For $b=6$, since $c=3$ is odd, $x_i=\pm \lceil c/2\rceil$ occurs only when $y_i=\pm c$.
In addition, we use the fact that if
$c$ (resp. $\overline{c}$) appears as a digit $y_i$, then it must be followed by
a non-positive (resp. non-negative) digit in the symmetric beta 
expansion by Lemma \ref{Adm}. Thus we may substitute (\ref{Key}) by
\begin{equation}
\label{Key3}
\left\lceil \frac c2 \right\rceil \left( 1+ \sum_{i=1}^{\infty} \frac 2{\alpha^{2i}} +\frac 2{\alpha} \right)= 
\left\lceil \frac c2 \right\rceil \left( 1+ \frac 2{\alpha^2-1} +\frac 2{\alpha} \right) 
\le c -\frac {10}{9\alpha}, 
\end{equation}
which holds for $b=4,6$. 
Therefore the statement for $b=6$ holds with $$
\kappa=\frac{c-10/(9\alpha)}{\alpha-\alpha_2}<\frac 12.$$
Finally we consider the case $b=4$. This implies $c=2$ and the same logic for $b=6$
does not work. We make a minor change of definition of $(s_n)$: 
$s_0=y_0$, $s_n=\lceil y_n/2 \rceil$ and $s_{-n}=\lfloor y_n/2 \rfloor$
for odd $n\in \N$, $s_n=\lfloor y_n/2 \rfloor$ and $s_{-n}=\lceil y_n/2 \rceil$
for even $n\in \N$. 
Then we can confirm
that $s_n\in \{-1,0,1\}$ and $s_n s_{n+1}\neq 11,\overline{1}\hspace{0.4mm}\overline{1}$ if $n\ge 1$
or $n\le -2$. Therefore we can apply (\ref{Key3}) and the same $\kappa$ for $b=4$ as well.

%One can confirm that the first digit of the expansion of $\eta\in [\kappa, 1/2)$ is $c$.

The proof goes in a similar manner for $\alpha=(b+\sqrt{b^2+4})/2$ and $\alpha_2=-1/\alpha$. 
In this case $\alpha\ge 3$ is equivalent to $b\ge 3$. 
There is a small
difference that if we take $\eta\in [\kappa,1/2)$, the value $(\alpha-\alpha_2)\eta/\alpha$ %$\frac{(\alpha-\alpha_2)\eta}{\alpha}$ 
may not be in $[-1/2,1/2)$
but in $[-(c+1/2)/\alpha, (c+1/2)/\alpha)$, which follows from a subtle\footnote{This is shown by classifying 
$b$ by its parity.} inequality
$$
\frac{\alpha-\alpha_2}{2}\le \left\lfloor \frac{\alpha+1}{2} \right\rfloor+\frac 12
$$
for $\alpha=(b+\sqrt{b^2+4})/2$ with $b\in \N$.
As described above, we obtain $d\left((\alpha-\alpha_2)\eta/\alpha \right)\in \D^{\N}$ 
%$d\left(\frac{(\alpha-\alpha_2)\eta}{\alpha}\right)\in \D^{\N}$
and it does not affect the course of the proof.
For $b\ge 7$, we see (\ref{Key}) holds and 
we can take
$$
\kappa=\frac {1}{\alpha-\alpha_2} \left(c - \frac c{\alpha}\right)=\frac {c(\alpha-1)}{\alpha^2+1}<\frac 12.
$$
%The expansion of $\eta\in [\kappa,1/2)$ starts with $c$ or $c-1$ but the proof goes in the same way.
For $b=6,4$, 
\begin{equation}
\label{Key4}
\left\lceil \frac c2 \right\rceil \left( 1+ \frac 2{\alpha-1} +\frac 1{\alpha} \right) 
\le c +\frac {2}{3\alpha},
\end{equation}
which gives the choice $\kappa= \left(c +  2/(3\alpha)\right)/(\alpha-\alpha_2)< 1/2$. 
%$\kappa=\frac {1}{\alpha-\alpha_2} \left(c + \frac 2{3\alpha}\right)< \frac 12$.
For $b=5$, since $c=3$ is odd, by using the same discussion as above, 
\begin{equation}
\label{Key5}
\left\lceil \frac c2 \right\rceil \left( 1+ \frac 2{\alpha^2-1} +\frac 2{\alpha} \right) 
\le c -\frac {c-1}{\alpha},
\end{equation}
which gives the choice 
$$
\kappa = \frac {1}{\alpha-\alpha_2} \left(c - \frac {c-1}{\alpha}\right)<\frac 12.
$$
Indeed, if $b=5$ then the digits $c$ (resp. $\overline{c}$)
must be followed by a non-positive (resp. non-negative) digit except for the first digit. 
Finally $b=3$ implies $c=2$ and we can use the same 
trick to change the definition of $(s_n)$ and the estimate (\ref{Key5}).
%For $b=4,6$, the expansion
%of Lemma \ref{Adm3} 
%comparing
%expansion of $1/2$ and $-1/2$, we notice in $d((\alpha-\alpha_2)\eta/\alpha)$, the digit 
%$c$ (resp. $\overline{c}$) can not be followed by a positive (resp. negative) digit 
%for any element $\eta\in [\kappa,1/2)$ except for the first digit.
\end{proof}

%Similar but more precise discussion shows
%the existence of an interval subset in $\LL(\alpha)$ for $b\ge 4$.

\section{Open problems}
% and related topics to $\LL(\alpha)$}
We list several open problems of interest. 
\begin{enumerate}
\item
It is of interest whether Theorems \ref{thm:main}, \ref{thm:main2} and \ref{Int}
can be extended to 
three remaining cases $\alpha\in \{(3+\sqrt{5})/2, (1+\sqrt{5})/2, 1+\sqrt{2} \}$
where $\alpha\le 3$. 
\item What can be said about $\LL(\alpha)$ in the case where $\alpha$ is a Pisot unit of degree greater than 2?
For example, describe the minimal $t>0$ such that $\\L(\alpha)\cap(0,t]\ne\emptyset$. Moreover, 
determine the minimal limit point $t_0(\alpha)$ of  $\LL(\alpha)$. In particular, is $t_0(\alpha)$ transcendental?
\item Let $\alpha>1$ be a fixed quadratic unit. Is $\mathrm{dim}_H (\LL(\alpha) \cap [0,t])$ continuous in $t$? 
\textcolor{red}{Note that Moreira \cite{Moreira.18} derived the continuity of $\mathrm{dim}_H (\L \cap [0,t])$ from
the dimension theory of sums of Cantor sets produced by non-essentially affine maps, but our setting is piecewise affine.}
\item Let $\alpha>3$ be a Pisot unit and $t_0(\alpha)$ be the minimal limit point of $\LL(\alpha)$. 
Can we find an interval in $[t_0(\alpha),1/2]\setminus \LL(\alpha)$? 
\item Let $\alpha$ be a Salem number. It is well known that 0 is a limit point of $\LL(\alpha)$. 
Is $\LL(\alpha)$ a closed set?
\item For a real number $c\in [0,1/2]$, put 
\[
\mathcal{G}(c):=\{\xi\in [0,1]\mid \|\xi \alpha^n\|\geq c\mbox{ for any }n\geq 0\}.
\]
What can be said on $\mathrm{dim}_H \mathcal{G}(c)$?  When $\alpha$
is an integer, Nilsson \cite{Nilsson} showed that $\mathrm{dim}_H \mathcal{G}(c)$ is continuous 
and its derivative is zero for almost every $c$.
\end{enumerate}

%Throughout this paper we discussed the maximal limit points of $\|\xi \alpha^n\|$. 
%We introduce related research on the initial values $\xi$ such that $\|\xi \alpha^n\|$ ($n\geq 0$) is restricted. 
%More precisely, let $a\geq 2$ be a fixed integer. For a real number $c$ with $0\leq c\leq 1/2$, put 
%\[
%\mathcal{G}(c):=\{\xi\in [0,1]\mid \|\xi a^n\|\geq c\mbox{ for any }n\geq 0\}.
%\]

\section*{Acknowledgments}

The authors are deeply indebted to Teturo Kamae and Shin-ichi Yasutomi for helpful and insightful discussions. 
We also would like to thank Yann Bugeaud and Dong Han Kim for informing us relevant references.
This research was partially supported by JSPS grants %(
(17K05159, 17H02849, BBD30028, 15K17505, 19K03439).

%\bibliographystyle{amsplain}
%\bibliography{../../reflist}

\bigskip
{\bf Appendix (Proof of Remark \ref{rem:ch})}
\bigskip

\begin{lem}
Let $\nu$ be a lower Christoffel word with $|\nu|\ge 2$ and $\x:=\nu^{\Z}$.
Assume that $\y(n):=\sigma^n(\x)=Avab\tilde{v}B$ 
satisfies $ab=\y(n)[0,1]\in\{01,10\}$, %, $\{a,b\}=\{0,1\}$ 
where $v\in \A^*$ and $A$, $B$ are left infinite and right infinite, respectively. 
Then $|v|=|\nu|-2$ if and only if $\y(n)[1,|\nu|]$ is equal to the lower or upper Christoffel word.
\end{lem}

\begin{proof}
We may assume $q:=|\nu|\ge 3$. For $p<q$ with $(p,q)=1$, consider a 
bi-infinite integer sequence $(n_i)_{i\in\Z}$ with $n_i\in [0,q-1]$ and 
$n_i \equiv i p \pmod {q}$. 
Define $$
a_i=\begin{cases} 0 & n_{i-1}<n_{i} \cr
                  1 & n_{i-1}>n_{i}\end{cases}.
$$
Then $a_1\dots a_{q}$ is nothing but the lower Christoffel word of slope $p/q$. 
Since switching the slope $p/q$ to $(q-p)/q$ corresponds to the involution $0\to 1,1\to 0$,
we may additionally assume that $2p<q$.
Then we have $a_{i}=1, a_{i+1}=0$ if and only if $n_i\in [0,p-1]$ and 
$a_{i}=0, a_{i+1}=1$ if and only if $n_i\in [q-p,q-1]$. 
For the index $k$ with $n_k=0$, we see that $a_{k-1}=1,a_{k}=0$ 
and $a_{k-1-\ell}=a_{k+\ell}$ holds for $\ell\le |\nu|-2$. 
This case corresponds to the lower
Christoffel word of slope $p/q$. 
When $n_k=q-1$, we have $a_{k-1}=0,a_{k}=1$ and 
$a_{k-1-\ell}=a_{k+\ell}$ holds for $\ell\le |\nu|-2$, corresponding to the upper Christoffel word. 
Our goal is to show that 
for an index $k$ that $n_k \not \equiv 0,q-1 \pmod{q}$ with $a_{k-1}
\neq a_k$,  there exists $\ell<|\nu|-2$ that $a_{k-1-\ell}\neq a_{k+\ell}$. 
We may assume that $2kp\not \equiv -1 \pmod{q}$. 
Indeed, $2kp\equiv -1 \pmod{q}$ implies $q$ is odd and $kp=(q-1)/2$, but
$p\le (q-1)/2 <q-p$ implies $a_{k-1}=a_k$.
Then we see that $n_{k+\ell}\in [0,p-1]$ is not equivalent to $n_{k-\ell}\in [q-p,q-1]$.
If they are equivalent, then there must exist $\ell$ that $n_{k+\ell}=0$ and $n_{k-\ell}=q-1$, which 
contradicts $2kp\not \equiv -1 \pmod{q}$.
Similarly $n_{k+\ell}\in [q-p,q-1]$ is not equivalent to $n_{k-\ell}\in [0,p-1]$.
Therefore there are $\ell_1,\ell_2 \in [1,q-1]\cap \Z$ 
that $n_{k+\ell_1}\in [0,p-1]$, $n_{k-\ell_1}\not \in [p-q,p-1]$,
$n_{k+\ell_2}\in [q-p,q-1]$ and 
$n_{k-\ell_2}\not \in [0,p-1]$.
%Since $0$ and $q-1$ are end points of the intervals,
%$n_{k+\ell_i}=0$ implies $n_{k-\ell_i}\not \in [p-q,p-1]$, and $n_{k+\ell_i}=q-1$ implies $n_{k-\ell_i}\not \in [0,p-1]$
%for $i=1,2$. 
Since one of $\ell_i$ is less than $q-1$, we obtain the result.
\end{proof}
\bigskip
{\bf Correction to "Theorem 2.2"}
\bigskip
\(\)\\
Theorem 2.2 was not correctly stated and should be replaced by
\bigskip

{\bf Theorem 2.2}
{\it
$$
\dim_H(\LL(a)\cap [0,t)) \le \frac{\log \left(2\lceil a^{\ell} t \rceil\right)
}{\log a^{\ell}}$$
for any integers $a\ge 2$ and $\ell\ge 1$. 
In particular, $\LL(a)$ has Lebesgue measure zero.
}
\bigskip

\begin{proof}
Let $t\in [0,1/2)$ and $\eta\in \LL(a)\cap [0,t]$. Then by taking a limsup word, we have shown 
that
$\{\eta a^n\} \in [0,t] \cup [1-t,1]$ for all $n\in \N$. 
However from this fact we can only show that
$\eta$ belongs to the attractor $Z$ of an iterated function system
$$
Z = \bigcup_{k\in Q} \left(\frac {Z}{a^{\ell}} + \frac {k}{a^{\ell}}\right)
$$
where $Q=\{q \in \Z\ |\ q \in [0,\lceil a^{\ell} t \rceil-1] 
\cup [\lfloor a^{\ell} (1-t) \rfloor, a^{\ell}-1] \}$. Since 
$\mathrm{Card}(Q)=\min\{a^{\ell},2\lceil a^{\ell} t \rceil\}$, we see
$$
\dim_{H}(Z)=\min\left\{1,\frac{\log \left(2\lceil a^{\ell} t \rceil\right)
}{\log a^{\ell}}\right\}
$$
and the required inequality.  If $t<1/2$ then taking $\ell$ with $\ell>\log(\frac2{1-2t})/\log a$ we have $\dim_{H}(Z)<1$. 
The final statement follows from
$$
\LL(a)-\left\{\frac 12\right\}=\bigcup_{n=1}^{\infty} \left(\LL(a)\cap \left[0,\frac 12- \frac 1n\right] \right).
$$
\end{proof}

One can derive a very precise estimate of $\dim_H(\LL(a)\cap [0,t))$ but we leave it for the other occasion.
\end{document}